%% file: main.tex
\documentclass[12pt]{article}

\usepackage[utf8]{inputenc}
\usepackage{titlesec}
\usepackage{graphicx}
\usepackage[a4paper,width=150mm,top=25mm,bottom=25mm,bindingoffset=6mm]{geometry}
\usepackage{fancyhdr}
\usepackage[english]{babel}
\usepackage[numbers]{natbib}
\usepackage{url}
\usepackage[hidelinks]{hyperref}
\usepackage[nodayofweek]{datetime}
\usepackage[all,cmtip]{xy}
\usepackage{comment}

\usepackage{physics}
\usepackage{amsthm}
\usepackage{mathtools}
\usepackage{amssymb}
\usepackage{bbm}
\usepackage{enumerate}
\usepackage[shortlabels]{enumitem}
\usepackage{tikz-cd}
\usepackage{datetime}
\usepackage{amsmath}
\usepackage{calligra}

\numberwithin{equation}{section}

\theoremstyle{definition}
\newtheorem{definition}{Definition}[section]
\theoremstyle{definition}

\theoremstyle{definition}
\newtheorem{remark}{Remark}[section]
\theoremstyle{remark}

\theoremstyle{plain}
\newtheorem{prop}{Proposition}[section]

\theoremstyle{plain}
\newtheorem{theorem}{Theorem}[section]
\theoremstyle{plain}
\newtheorem{cor}{Corollary}[section]
\theoremstyle{plain}
\newtheorem{lemma}{Lemma}[section]

\theoremstyle{plain}

\newtheorem*{conjecture}{Conjecture}
\theoremstyle{definition}
\newtheorem*{results}{Main results}

\newtheorem*{aknowledgments}{Aknowledgments}

\newcommand{\Z}{\mathbb{Z}}
\newcommand{\N}{\mathbb{N}}
\newcommand{\C}{\mathbb{C}}
\newcommand{\Q}{\mathbb{Q}}
\newcommand{\CP}{\mathbb{P}}

\newcommand{\A}{\mathbb{A}}

\newcommand{\sheaf}[1]{\mathcal{#1}}

\newcommand{\mf}{\mathcal{M}}

\newcommand{\grothgrp}[1]{K(#1)}
\newcommand{\SL}[1]{\textup{SL}(#1)}

\newcommand{\grothnum}[1]{\grothgrp{#1}_\text{num}}
\newcommand{\schC}{(\text{Sch}^{ft}/\C)}

\newcommand{\Rss}{R^{ss}}

\newcommand{\Proj}[1]{\textup{Proj}#1}

\DeclareMathOperator{\Soc}{Soc}
\DeclareMathOperator{\Pic}{Pic}
\DeclareMathOperator{\Coker}{Coker}
\DeclareMathOperator{\Ker}{Ker}

\DeclareMathOperator{\Hom}{Hom}
\DeclareMathOperator{\sheafHom}{\sheaf{H}\!\textit{om}\,}

\DeclareMathOperator{\Ext}{Ext}

\DeclareMathOperator{\hd}{hd}
\DeclareMathOperator{\rk}{rk}
\DeclareMathOperator{\Exts}{\mathcal{E}\!\mathit{xt}}
\DeclareMathOperator{\depth}{depth}

\DeclareMathOperator{\length}{length}
\DeclareMathOperator{\codim}{codim}

\DeclareMathOperator{\Spec}{Spec}
\DeclareMathOperator{\Quot}{Quot}

\DeclareMathOperator{\Supp}{Supp}
\DeclareMathOperator{\Grass}{Grass}
\DeclareMathOperator{\Coh}{Coh}

\DeclareMathOperator{\gr}{gr}
\DeclareMathOperator{\End}{End}

\newcommand{\Address}{{
  \bigskip
  \footnotesize

  \textsc{Universit\'e de Lorraine, IECL, F-54000 Nancy, France}\par\nopagebreak
  \textit{E-mail address}: \texttt{mihai-cosmin.pavel@univ-lorraine.fr}
}}
  
\begin{document}

\title{Restriction theorems for semistable sheaves}
\author{Mihai Pavel}
\date{}
\maketitle
\input{tex/Abstract}
\input{tex/Introduction}
\input{tex/Preliminaries}

\input{tex/RestrictionTheorems}
\input{tex/ModuliSpace}
\bibliography{bib/mainbib}

\Address
\end{document}

%% file: tex/Abstract.tex
\begin{abstract}
In this paper we prove restriction theorems for torsion-free sheaves that are (semi)stable with respect to the truncated Hilbert polynomial over a smooth projective variety. Our results apply in particular to Gieseker-semistable sheaves and generalize the well-known restriction theorems of Mehta and Ramanathan. As an application we construct a moduli space of sheaves in higher dimensions.
\end{abstract}

%% file: tex/Introduction.tex
\section{Introduction}

In the moduli theory of sheaves, the so-called restriction theorems for slope semistability play a particularly important role. In many cases they provide a useful instrument to reduce the study of sheaves to lower dimensions. More exactly, they ensure that slope-semistability (or slope-stability) is preserved by restriction to a general divisor of sufficiently large degree on a smooth projective variety. For example, they were used to show Bogomolov's inequality in higher dimensions (see \cite[Thm.~7.3.1]{huybrechts2010geometry}). Another noteworthy application was in the construction of moduli spaces of slope-semistable sheaves over surfaces by Le Potier \cite{le1992fibre} and Jun Li \cite{li1993algebraic}.

At this point there are several different proofs of the restriction theorems in the literature. The first general results were given by Mehta and Ramanathan \cite{mehta1982semistable, mehta1984restriction}. In zero characteristic, Flenner later proved in \cite{flenner1984restrictions} an effective restriction theorem but only for slope-semistability. He found a bound on the degree of the divisor that depends only on the numerical invariants of the sheaf. Finally, by a different approach, Langer gave in \cite{langer2004semistable} stronger effective restriction theorems that work also in positive characteristic. 

\begin{results}
Let $X$ be a smooth projective variety over an algebraically closed field $k$, and fix $\sheaf{O}_X(1)$ a very ample line bundle. In this paper we prove restriction theorems for torsion-free sheaves on $X$ that are semistable with respect to the truncated Hilbert polynomial. We call such sheaves $\ell$-\textit{semistable} (see Definition~\ref{ssDef}), where $\ell$ marks the level of truncation. In particular one recovers the notion of slope-(semi)stability for $\ell=1$ and of Gieseker-(semi)stability for $\ell=\dim(X)$. Our aim is to prove the following two theorems, thus answering in the affirmative a conjecture posed by Langer in \cite[Conj.~3.13]{langer2009moduli}.

\begin{theorem}\label{thm:GiesekerRestriction}
Let $\ell < \dim(X)$ and $E$ be an $\ell$-semistable torsion-free sheaf on $X$. The restriction $E|_D$ to a general divisor $D \in |\sheaf{O}_X(a)|$ remains $\ell$-semistable for $a \gg 0$.
\end{theorem}

\begin{theorem}\label{thm:restrictionStable}
Let $\ell < \dim(X)$ and $E$ be an $\ell$-stable torsion-free sheaf on $X$. The restriction $E|_D$ to a general divisor $D \in |\sheaf{O}_X(a)|$ remains $\ell$-stable for $a \gg 0$.
\end{theorem}

As a consequence of Theorem~\ref{thm:GiesekerRestriction}, we also obtain a restriction theorem for Gieseker-semistable sheaves. However the range for $\ell$ is sharp in Theorem~\ref{thm:restrictionStable}. Indeed, the result is false for a Gieseker-stable but not slope-stable sheaf on a surface. For a different example see also \cite[Exam.~3.1]{toma2017continuity}.

Our proof is by induction on $\ell$ and follows the strategy applied by Mehta and Ramanathan \cite{mehta1982semistable, mehta1984restriction} in the case of slope-semistability. However, these methods give no effective bound on the degree of the divisor from which the restriction remains $\ell$-semistable. Finding such effective bounds may be an incentive for future studies. One may also ask whether the theorems hold for pure sheaves supported in positive codimension. In this regard, we proved in \cite{pavel2021moduli} an effective restriction theorem for slope-(semi)stable pure sheaves, by following a different approach due to Langer \cite{langer2004semistable}. We expect the result to hold more generally for $\ell$-semistable pure sheaves:

\begin{conjecture}
Let $E$ be an $\ell$-(semi)stable pure sheaf of dimension $d$ on $X$ with $\ell < d$. The restriction $E|_D$ to a general divisor $D \in |\sheaf{O}_X(a)|$ remains $\ell$-(semi)stable for $a \gg 0$.
\end{conjecture}

Our restriction theorems may find applications in the moduli theory of sheaves. To give an example already envisioned by Le Potier \cite{le1992fibre}, one may use the restriction theorems to construct moduli spaces of $\ell$-semistable sheaves on $X$. In this respect, here we give the construction when $\ell = \dim(X)-1$, to obtain a moduli space of sheaves in higher dimensions over the field of complex numbers. We recover a result due to Huybrechts--Lehn \cite[Sect.~8]{huybrechts2010geometry} when $X$ is a surface. By similar methods, Greb and Toma \cite{greb2017compact} constructed moduli spaces of slope-semistable sheaves in higher dimensions, but they had to work over the category of weakly normal varieties. In contrast to their result, we need not to impose such restrictive assumptions on schemes (compare Theorem~\ref{thm:mainThm} with \cite[Main Thm.]{greb2017compact}), but on the other hand we have a stronger stability condition on sheaves. Furthermore, in our case we are able to give a complete description of the geometric points of the moduli space (see Theorem~\ref{thm:geometricSep}).

The moduli spaces of slope-semistable sheaves constructed by Huybrechts--Lehn and Greb--Toma are closely related to the Donaldson--Uhlenbeck compactification of Hermitian Yang--Mills (HYM) connections (see \cite{li1993algebraic} and \cite{greb2018complex}). This is partly due to the so-called Kobayashi--Hitchin correspondence, which says that vector bundles over compact complex manifolds admit HYM connections (or Hermite--Einstein metrics) if and only if they are slope-polystable. There is an analogue correspondence between Gieseker-stable sheaves and almost Hermite--Einstein metrics (see \cite{leung97einstein}). For this reason, we expect that the moduli space we construct here may find applications in the context of gauge theory.
\end{results}

\begin{aknowledgments}
I would like to thank my advisor Matei Toma for his patient guidance, insightful suggestions and encouragement during my PhD studies.
\end{aknowledgments}

%% file: tex/Preliminaries.tex
\section{Preliminaries}\label{sect:preliminaries}

In this section we set up notation and terminology. First we introduce the notion of $\ell$-(semi)stability and recall some of its basic properties, following \cite[Sect.~1]{huybrechts2010geometry}.

Throughout this paper, let $X$ be a smooth integral projective scheme of dimension $n$ over a field $k$, and $\sheaf{O}_X(1)$ a very ample line bundle on $X$. For the moment $k$ is a field of arbitrary characteristic, which is not necessarily algebraically closed.

Let $E$ be a pure sheaf of dimension $d$ on $X$, i.e.~any non-zero subsheaf $F \subset E$ has dimension $d$. The Hilbert polynomial of $E$ is given by its Euler characteristic $P(E,m) \coloneqq \chi(X, E \otimes \sheaf{O}_X(m))$ and has the form
\begin{align*}
    P(E,m) = \sum_{j=0}^d \alpha_{d-j}(E) \frac{m^{d-j}}{(d-j)!},
\end{align*}
with rational coefficients $\alpha_{d-j}(E)$. In what follows, a pure sheaf on $X$ of dimension $n = \dim(X)$ will be called \textit{torsion-free}.

\begin{definition}\label{ssDef}
Let $E$ be a coherent sheaf of dimension $d$ on $X$ and take $1 \leq \ell \leq d$. Then $E$ is called \textit{$\ell$-semistable} (resp. \textit{$\ell$-stable}) if it is pure and
    \begin{align*}
    \left(\frac{\alpha_{d-1}(F)}{\alpha_d(F)},\ldots,\frac{\alpha_{d-\ell}(F)}{\alpha_d(F)}\right) \leq \left(\frac{\alpha_{d-1}(E)}{\alpha_d(E)},\ldots,\frac{\alpha_{d-\ell}(E)}{\alpha_d(E)}\right) \quad (\textnormal{resp.~}<)
\end{align*}
for all subsheaves $F \subset E$ such that $0 < \alpha_d(F) < \alpha_d(E)$, where $\leq$ is the lexicographic order.
\end{definition}

With this definition we recover the notion of slope-(semi)stability for $\ell=1$, and of Gieseker-(semi)stability for $\ell=d$. More generally, one may define $\ell$-stability in the quotient category $\Coh_d(X)/\Coh_{d-\ell-1}(X)$ (see \cite[Sect.~1.6]{huybrechts2010geometry}), where $\Coh_d(X)$ is the category of coherent sheaves of dimension $\leq d$. However, for the limited purpose of this paper, we restrict to work within the category $\Coh(X)$.

Sometimes it is convenient to formulate the $\ell$-semistability conditions in terms of truncated Hilbert polynomials. If $\Q[T]_j$ denotes the vector subspace of polynomials of degree $\leq j$ in $\Q[T]$, let $P_{\ell}(E) \in \Q[T]_d/\Q[T]_{d-\ell-1}$ be the $\ell$-truncation of the Hilbert polynomial of $E$ and $p_{\ell}(E) \coloneqq P_{\ell}(E)/\alpha_d(E)$ its reduced form. Then, clearly, $E$ is $\ell$-semistable (resp.~$\ell$-stable) if and only if
\begin{align*}
    p_{\ell}(F) \leq p_{\ell}(E) \quad ( \text{resp.}~<)
\end{align*}
for all subsheaves $F \subset E$ such that $0 < \alpha_d(F) < \alpha_d(E)$. Above we compare the polynomials with respect to the natural lexicographic order of their coefficients. 

\subsection*{Properties of $\ell$-semistability}

Most of the basic properties of Gieseker-(semi)stability, such as the ones treated in \cite[Sect.~1]{huybrechts2010geometry}, hold also for $\ell$-(semi)stability with almost identical proofs. Furthermore, one can show by using known techniques that families of $\ell$-semistable sheaves satisfy the following properties: boundedness, openness and properness, which are particularly important for studying moduli of sheaves. 

\begin{lemma}\label{lemma:boundedness}
Assume that $k$ is algebraically closed. The family of $\ell$-semistable sheaves of fixed Hilbert polynomial on $X$ is bounded.
\end{lemma}
\begin{proof}
See \cite[Thm.~4.4]{langer2004semistable}.
\end{proof}

\begin{lemma}\label{lemma:openness}
The property of being $\ell$-semistable, resp. simple, geometrically $\ell$-stable, is open in flat families of sheaves. 
\end{lemma}
\begin{proof}
The proof is adapted from \cite[Prop.~2.3.1]{huybrechts2010geometry}. The property of being simple, i.e. $\dim_{k(s)}(\End(\sheaf{F}_s))=1$, is open by the semicontinuity properties of the relative Ext sheaves \cite[Satz~3]{banica80variation}. We treat next the case of $\ell$-semistability, see the cited reference for the remaining case.

Let $S$ be a $k$-scheme of finite type and $\sheaf{F}$ an $S$-flat family of sheaves on $X$ with Hilbert polynomial $P$. Since the property of being pure is open in flat families, we may assume the fibers $\sheaf{F}_s$ are pure of dimension $d \coloneqq \deg(P)$.
We first show that all destabilizing quotients $\sheaf{F}_s \to F'$ fit inside a proper Quot scheme. Consider the following set
\begin{align*}
    H = \{ &P(F') \in \Q[T] : \sheaf{F}_s \to F' \textnormal{ is a pure quotient of dimension $d$} \\
    {}&\text{for a closed point $s \in S$ such that $p_{\ell}(F') <  p_{\ell}$}\}.
\end{align*}
where $p_{\ell}$ denotes the reduced $\ell$-truncation of $P$.
By a boundedness result due to Grothendieck (see \cite[Lem.~1.7.9]{huybrechts2010geometry}) we know that $H$ is finite. Consider the relative Quot scheme $\varphi: Q \coloneqq \Quot(\sheaf{F},H) \to S$ of quotients $\sheaf{F}_s \to F'$ with $P(F') \in H$. Since $Q$ is proper, its scheme-theoretic image $\varphi(Q)$ is a closed subset of $S$.
Then a fiber $\sheaf{F}_s$ is $\ell$-semistable if and only if $s$ is not contained in $\varphi(Q)$. 
\end{proof}

\begin{lemma}\label{lemma:properness}
Assume that $k$ is algebraically closed. Let $A$ be a discrete valuation ring with residue field $k$ and quotient field $K$. Let $E$ be an $A$-flat family of $d$-dimensional sheaves on $X$ such that $E_K = E \otimes_A K$ is $\ell$-semistable on $X_K$. Then there exists a subsheaf $F \subset E$ such that $F_K = E_K$ and $F_k$ is $\ell$-semistable on $X$. 
\end{lemma}
\begin{proof}
In the slope-semistable case this result is due to Langton \cite{langton1975valuative}. In fact her proof can be adapted to work also for $\ell$-semistability (see \cite[Thm.~2.B.1]{huybrechts2010geometry}).
\end{proof}

\begin{definition}
Let $E$ be a coherent sheaf on $X$. Then $E$ is called $\ell$-\textit{simple} if $\End(E|_{X \setminus Y}) \cong k$ for every closed subset $Y \subset X$ such that $\codim(\Supp(E) \cap Y,\Supp(E)) \geq \ell + 1$. 
\end{definition}

\begin{remark}
Clearly any $\ell$-simple sheaf is also $(\ell + 1)$-simple. Thus any $\ell$-simple sheaf $E$ is in particular simple, i.e. $\End(E) \cong k$. Also, one easily checks that any $\ell$-stable sheaf is $\ell$-simple.
\end{remark}

\begin{lemma}\label{lemma:simpleDepth}
Let $E$ be a torsion-free sheaf on $X$. Then there exists a unique coherent sheaf $E \subset G$ on $X$ such that
\begin{enumerate}[(1),nolistsep]
  \item $E$ and $G$ are isomorphic in codimension $\ell$,
  \item $G$ satisfies
   \begin{align}\label{eq:depthProp}
  \depth_x(G) \geq 2 \text{ for every } x \in X \text{ such that } 		\dim(\sheaf{O}_{X,x})\geq \ell+1 \tag{$\star$}
  \end{align}
  \item for every closed subset $Y \subset X$ such that $\codim(Y,X) \geq \ell + 1$,
\begin{align*}
  \Hom(G,G) \to \Hom(G|_{X \setminus Y},G|_{X \setminus Y})
\end{align*}
is an isomorphism.
\end{enumerate}
\end{lemma}
\begin{proof}
Let $F \coloneqq E^{\vee \vee}$ be the double dual of $E$ and set $Q \coloneqq F/E$. Denote by $T$ the torsion part of $Q$ supported in codimension $\geq \ell+1$, and take $G \subset F$ such that $G/E \cong T$. Then $\depth_x(Q/T) \geq 1$ for every $x \in X$ such that $\dim(\sheaf{O}_{X,x})\geq \ell + 1$. From the short exact sequence
\begin{align*}
    0 \to G \to F \to Q/T \to 0,
\end{align*}
we have by \cite[Tag 00LX]{stacks-project}
\begin{align*}
  \depth_x(G) \geq \min\{\depth_x(F),\depth_x(Q/T) + 1\}
\end{align*}
for every point $x \in X$. But $F$ is reflexive, and so $\depth_x(F) \geq 2$ whenever $\dim(\sheaf{O}_{X,x})\geq 2$. Hence $G$ satisfies \eqref{eq:depthProp}. Also, note that $E$ and $G$ are isomorphic in codimension $\ell$, outside the support of $T$.

It remains to prove $(3)$. Let $Y \subset X$ be a closed subset as in the statement. As $\depth_Y(G) \geq 2$, by a standard result in local cohomology (see \cite[Thm.~3.8]{grothendieck1967local}) we have $\Exts_Y^i(G,G)=0$ for $i=0,1$. This further implies
\begin{align*}
 \Hom_Y(G,G) = \Ext_Y^1(G,G)=0,
\end{align*}
by using the spectral sequence 
\begin{align*}
  E_2^{pq} = H^p(X,\Exts_Y^q(G,G)) \Rightarrow \Ext_Y^{p+q}(G,G)
\end{align*}
given in \cite[VI, Thm.~1.6]{sga2}. From the exact sequence (see \cite[VI, Cor.~1.9]{sga2})
\begin{align*}
 \Hom_{Y}(G,G) \to \Hom(G,G) \to \Hom(G|_{X \setminus Y},G|_{X \setminus Y})  \to \Ext_{Y}^1(G,G)
\end{align*}
it follows that
\begin{align*}
  \Hom(G,G) \to \Hom(G|_{X \setminus Y},G|_{X \setminus Y}) 
\end{align*}
is an isomorphism.
\end{proof}

\begin{lemma}\label{lemma:openSimple}
Let $S$ be an integral scheme of finite type over $k$ and $\xi \in S$ its generic point. Let $f: Z \to S$ be a smooth projective morphism and $\sheaf{F}$ an $S$-flat family of torsion-free sheaves such that $\sheaf{F}_s$ is $\ell$-simple for every closed point $s \in S$. Then $\sheaf{F}_\xi$ is also $\ell$-simple.
\end{lemma}
\begin{proof}
Let $\iota : \sheaf{F}_\xi \to \sheaf{G}_\xi$ be the inclusion given by Lemma~\ref{lemma:simpleDepth}, in particular $\sheaf{F}_\xi$ and $G_\xi$ are isomorphic in codimension $\ell$. By shrinking $S$ if necessary, we can extend $\iota$ to an inclusion $\sheaf{F} \subset \sheaf{G}$ such that $\sheaf{G}$ is an $S$-flat family and moreover, $\sheaf{F}_s$ and $\sheaf{G}_s$ are isomorphic in codimension $\ell$ for every $s \in S$. Note that $\sheaf{F}_s$ is $\ell$-simple if and only if $\sheaf{G}_s$ is $\ell$-simple, hence it is enough to prove the statement for the family $\sheaf{G}$.

Let $Y_\xi \subset Z_\xi$ be a closed subset of codimension $\geq \ell + 1$ in $Z_\xi$. By Lemma~\ref{lemma:simpleDepth} there is an isomorphism
\begin{align*}
  \Hom(\sheaf{G}_\xi,\sheaf{G}_\xi) \to \Hom(\sheaf{G}_\xi|_{Z_\xi \setminus Y_\xi},\sheaf{G}_\xi|_{Z_\xi \setminus Y_\xi}).
\end{align*}
By the semicontinuity properties of the relative Ext sheaves \cite[Satz~3]{banica80variation}, it follows that $\sheaf{G}_\xi$ is simple (see also Lemma~\ref{lemma:openness}). We conclude that $\sheaf{G}_\xi$ is $\ell$-simple.
\end{proof}

We next recall the existence of the Harder-Narasimhan filtration. The result is standard and we omit its proof.
 
\begin{lemma}
If $E$ is a pure sheaf on $X$, then $E$ has a unique $\ell$-Harder-Narasimhan filtration
\begin{align*}
    0 = E_0 \subset E_1 \subset \ldots \subset E_m = E
\end{align*}
such that its factors $E_i/E_{i-1}$ are $\ell$-semistable and satisfy
\begin{align*}
    p_\ell(E_1) > p_{\ell}(E_2/E_1) > \ldots > p_{\ell}(E/E_{m-1}).
\end{align*}
We call $E/E_{m-1}$ the minimal $\ell$-destabilizing quotient of $E$.
 \end{lemma}
 
We also have a relative version of the Harder-Narasimhan filtration, which will provide a key ingredient in proving Theorem~\ref{thm:GiesekerRestriction}.

\begin{lemma}\label{lemma:relativekHN}
Let $S$ be an integral scheme of finite type over $k$ and $\sheaf{F}$ an $S$-flat family of $\ell$-semistable sheaves on $X$. Then there exists a dense open subset $U \subset S$ and a filtration
\begin{align*}
    0 = \sheaf{F}_0 \subset \sheaf{F}_1 \subset \ldots \subset \sheaf{F}_m = \sheaf{F}
\end{align*}
such that
\begin{enumerate}[nolistsep]
    \item its factors $\sheaf{F}_i/\sheaf{F}_{i-1}$ are flat over U,
    \item for each closed point $s \in U$, the restriction $0 = \sheaf{F}_{0,s} \subset \sheaf{F}_{1,s} \subset \ldots \subset \sheaf{F}_{m,s} = \sheaf{F}_s$ is the $\ell$-Harder-Narasimhan filtration of $\sheaf{F}_s$.
\end{enumerate}
\end{lemma}
\begin{proof}
The proof is very similar to that in \cite[Thm.~2.3.2]{huybrechts2010geometry}.
\end{proof}

If $E$ is already $\ell$-semistable, then $E$ has a so-called Jordan-H\"older filtration
\begin{align*}
    0 = E_0 \subset E_1 \subset \ldots \subset E_m = E
\end{align*}
such that its factors $E_i/E_{i-1}$ are $\ell$-stable with $p_\ell(E_i/E_{i-1})=p_\ell(E)$. This filtration is not unique in general, but nevertheless its graded sheaf $\gr_\ell(E) = \oplus_i{E_i/E_{i-1}}$ is so in codimension $\ell$, i.e. outside a closed subset of codimension $\geq \ell+1$ in $\Supp(E)$. 

As in the case of Gieseker-semistability, $E$ has a unique extended socle, that we describe precisely below. The following is adapted from \cite[Sect.~1.5]{huybrechts2010geometry} and treats the case of $\ell$-stability. A similar discussion for slope-stability can be found in \cite[Sect.~6]{gangopadhyaymehta}. We define the \textit{$\ell$-socle} $\Soc(E)$ of $E$ as the saturation of the sum of all $\ell$-stable subsheaves $F \subset E$ with $p_{\ell}(F) = p_{\ell}(E)$. 

\begin{definition}\label{def:extendedSocle}
The \textit{extended $\ell$-socle} of an $\ell$-semistable sheaf $E$ is the maximal subsheaf $F \subset E$ with $p_{\ell}(F)=p_{\ell}(E)$ such that every graded factor of $\gr_\ell(F)$ is isomorphic with a graded factor of $\gr_\ell(\Soc(E))$ in codimension $\ell$.
\end{definition}

\begin{lemma}\label{lem:algSocle}
Let $E$ be an $\ell$-stable sheaf on $X$ and $k \subset K$ the algebraic closure. Then $E_K = E \otimes_k K$ equals its extended $\ell$-socle.
\end{lemma}
\begin{proof}
The proof is identical to that of \cite[Lem.~6.2]{gangopadhyaymehta}.
\end{proof}

\begin{lemma}\label{lem:extendedSocle}
Let $E$ be a coherent sheaf on $X$. If $E$ is $\ell$-simple, $\ell$-semistable and equals its extended $\ell$-socle, then $E$ is $\ell$-stable.
\end{lemma}
\begin{proof}
Suppose that $E$ is not $\ell$-stable, so there exists an $\ell$-stable quotient $E \to F$ with $p_\ell(F) = p_\ell(E)$. As $E$ equals its extended $\ell$-socle, there is an open subset $U \subset \Supp(E)$ with $\codim(\Supp(E) \setminus U, \Supp(E)) \geq \ell+1$ such that $F|_U \subset E|_U$. This gives a non-trivial composition $E|_U \to F|_U \to E|_U$. But $E$ is $\ell$-simple, so $F|_U \cong E|_U$, which yields a contradiction.
\end{proof}

\begin{lemma}\label{lem:geomSt}
If $E$ is $\ell$-simple, then $E$ is $\ell$-stable if only if $E$ is geometrically $\ell$-stable.
\end{lemma}
\begin{proof}
This follows immediately from Lemma~\ref{lem:algSocle} and Lemma~\ref{lem:extendedSocle}. 
\end{proof}

\subsection*{The Hilbert polynomial}

Below we relate the coefficients of the Hilbert polynomial of a sheaf $E$ to those of its restriction $E|_D$ to a divisor. For $a > 0$, let $\Pi_a \coloneqq |\sheaf{O}_X(a)|$ be the linear system of hypersurfaces of degree $a$ on $X$. For a pure sheaf $E$ of dimension $d$ on $X$, consider the short exact sequence
\begin{align*}
    0 \to E(-a) \to E \to E|_D \to 0,  
\end{align*}
for a divisor $D \in \Pi_a$ avoiding the associated points of $E$. By the additivity of the Hilbert polynomial in short exact sequences we obtain
\begin{align}\label{eq:alphaComputations}
    \alpha_{d-j}(E|_D) ={}& \alpha_{d-j}(E) - \alpha_{d-j}(E(-a)) \notag \\
    ={}& \sum_{l=1}^{j} \frac{(-1)^{l-1}a^l}{l!}\alpha_{d-j+l}(E)
\end{align}
for $1 \leq j \leq d$. 

\begin{lemma}\label{cor:destabilizing}
Let $E$ and $F$ be pure sheaves of dimension $d$ on $X$ and $D \in \Pi_a$ a divisor such that $E|_D$ and $F|_D$ are still pure. If $\ell < d$, then $p_{\ell}(E|_D) \geq p_{\ell}(F|_D)$ (resp.~$>$) if and only if $p_{\ell}(E) \geq p_{\ell}(F)$ (resp.~$>$).
\end{lemma}
\begin{proof}
The inequalities follow in an iterative manner from \eqref{eq:alphaComputations}.
\end{proof}

Note that \eqref{eq:alphaComputations} gives us no information about the free coefficient $\alpha_0(E)$, which explains the upper bound $\ell < d$ in the statement above. 

For $1 \leq j \leq d$ and $D \in \Pi_a$, set
\begin{align*}
    \beta_{d-j}(E|_D) \coloneqq \alpha_{d-j}(E|_D) + \frac{a}{2}\alpha_{d-j+1}(E|_D) + \sum_{l=1}^{j-1}\frac{(-1)^{l-1}B_{2l}}{(2l)!}a^{2l}\alpha_{d-j+1+l}(E|_D),
\end{align*}
where $B_{2l}$ denotes the $l$-th Bernoulli number. Recall that the Bernoulli numbers can be defined by the generating power series
\begin{align*}
    Q(x) = \frac{x}{1-e^{-x}} = 1 + \frac{x}{2} + \sum_{l \geq 1}\frac{(-1)^{l-1}B_{2l}}{(2l)!}x^{2l} = 1 + \frac{x}{2} + \frac{x^2}{12} - \frac{x^4}{720} + \ldots
\end{align*}
By using its inverse
\begin{align*}
    Q(x)^{-1} = \frac{1-e^{-x}}{x} =  \sum_{l \geq 1} \frac{(-1)^{l-1}x^{l-1}}{l!},
\end{align*}
one can derive
\begin{align}\label{eq:bernoulliInv}
  \alpha_{d-j}(E|_D) = \sum_{l=1}^{j} \frac{(-1)^{l-1}a^{l-1}}{l!} \beta_{d-j+l-1}(E|_D).
\end{align}
A direct comparison of \eqref{eq:alphaComputations} and \eqref{eq:bernoulliInv} shows that $\beta_{d-j}(E|_D)/a = \alpha_{d-j+1}(E)$ whenever $D$ avoids the associated points of $E$.

\subsection*{Determinant line bundles}

The following discussion will serve only in Section~\ref{sect:moduliSpace}, so the reader may skip this part and come back when needed.

Fix a numerical Grothendieck class $c \in \grothnum{X}$. Let $S$ be a scheme over $k$ and $\sheaf{E}$ an $S$-flat family of coherent sheaves of class $c$ on $X$. Consider the map $\lambda_\sheaf{E} : K(X) \to \Pic(S)$ given by the composition:
    \begin{align}\label{eq:determinant}
        K(X) \xrightarrow{p_2^*} K^0(S \times X) \xrightarrow{\cdot [\sheaf{E}]} K^0(S \times X) \xrightarrow{{p_1}_{!}} K^0(S) \xrightarrow{\det} \Pic(S),
    \end{align}
which associates to a class $u \in K(X)$ its corresponding determinant line bundle $\lambda_\sheaf{E}(u)$ over $S$. One notices that this construction behaves well with respect to base change. We refer the reader to \cite{le1996module, le1992fibre} for an account of the properties of $\lambda_\sheaf{E}$. 

The Grothendieck group $K(X)$ is endowed with a natural quadratic form given by the Euler characteristic $\chi$. If $A \subset K(X)$ is a subset, let $A^\perp \subset K(X)$ be the orthogonal complement of $A$ with respect to $\chi$. Let $h$ denote the class of $\sheaf{O}_X(1)$ in $K(X)$. If $u \in K(X)$ is a class such that $u \in c^\perp \cap \{1,h,\ldots,h^n\}^{\perp \perp}$ then there is a natural determinant line bundle $\lambda(u)$ over the moduli space $M^{Gss}(c)$ of Gieseker-semistable sheaves of class $c$ on $X$ (see \cite[Thm.~8.1.5]{huybrechts2010geometry}). 

\subsection*{$G$-properness and Iitaka $G$-fibration}

In this part we work over the category $\schC$ of schemes of finite type over $\C$. Fix $G$ a connected algebraic group over $\C$.

\begin{definition}\label{def:Gproper}
Let $R$ be a $G$-scheme and $f : R \to S$ a $G$-invariant separated morphism. We say that $f$ is $G$-\textit{proper} if for every commutative diagram
\begin{center}
\begin{tikzcd}[column sep=normal]
\Spec(K) \ar[d,"j"] \ar[r,"g"] &  R \ar[d,"f"] \\  \Spec(A) \ar[ru,dashed,"i"] \ar[r] & S
\end{tikzcd}
\end{center}
where $A$ is a discrete valuation ring over $\C$ of quotient field $K$, there exists a morphism $i : \Spec(A) \to R$ such that $i \circ j$ and $g$ differ by a group element in $G(K)$.
\end{definition}

Let $R$ be a $G$-proper scheme over $\C$ endowed with a $G$-equivariant line bundle $\sheaf{L}$ that is $G$-semiample, i.e. $\sheaf{L}$ is globally generated by $G$-invariant sections. If $G$ is reductive, then one can form the Iitaka $G$-fibration of $(R,\sheaf{L})$ as in \cite[Sect.~5]{pavel2021moduli}, which is given by a Proj scheme
\begin{align*}
\Proj \bigoplus_{k\geq 0} H^0(R,\sheaf{L}^{ke})^G
\end{align*}
for some $e > 0$. In Section~\ref{sect:moduliSpace}, we will use the Iitaka $G$-fibration to construct a moduli space of sheaves.

%% file: tex/RestrictionTheorems.tex
\section{Restriction theorems for $\ell$-semistability}\label{section:RestrictionTheorems}

In this section we assume that the base field $k$ is algebraically closed of arbitrary characteristic. We will show both restriction theorems by following the methods of Mehta and Ramanathan. For this reason, in some situations the proofs follow a similar pattern. Whenever this happens we only sketch the arguments and refer the reader to \cite{mehta1982semistable, mehta1984restriction} or \cite[Sect.~7]{huybrechts2010geometry} for further details. We also recommend the lecture notes in \cite{gangopadhyaymehta} for a detailed exposition of the slope-(semi)stable case.

We first show Theorem~\ref{thm:GiesekerRestriction}. The basic idea of the proof is to extend a destabilizing quotient $E|_D \to F_D$ over $X$, since then Lemma~\ref{cor:destabilizing} will give a contradiction. We argue by induction on $\ell$. We already know by the Mehta--Ramanathan restriction theorem that the result holds for $\ell = 1$, so it remains to prove the induction step. 

Fix $1 < \ell < n$ and let $E$ be an $\ell$-semistable torsion-free sheaf on $X$. We want to vary the minimal $\ell$-destabilizing quotient of $E|_D$ in a flat family over $\Pi_a$. For this consider the incidence variety $Z_{a} = \{ ([D],x) \in \Pi_a \times X  | x \in D \}$ with its natural projections:
\begin{center}
\begin{tikzcd}[column sep=normal]
Z_{a} \ar[d,"p"] \ar[r,"q"] &  X \\  \Pi_a
\end{tikzcd}
\end{center}

\begin{remark}\label{rmk:PicDec}
If $\sheaf{K}$ denotes the kernel of the evaluation map
\begin{align*}
    H^0(X,\sheaf{O}_X(a)) \otimes \sheaf{O}_X \to \sheaf{O}_X(a), 
\end{align*}
then $q: Z_{a} \to X$ is in fact isomorphic to the projective bundle $\CP(\sheaf{K}^\vee) \to X$. Also, there exists an isomorphism $\Pic(Z_a) \cong \Pic(X) \times \Pic(\Pi_a)$ such that $\sheaf{O}_{\CP(\sheaf{K}^\vee)}(1) \cong p^*(\sheaf{O}_{\Pi_a}(1))$ (see the proof in \cite[Prop.~2.1]{mehta1982semistable}).
\end{remark}

As a matter of notation, we will denote by $[D] \in \Pi_a$, or simply by $D \in \Pi_a$, the closed point of $\Pi_a$ coressponding to an effective divisor $D \subset X$. Also we will often identify the fibre $p^{-1}([D])$ with $D \subset X$ via $q$.

By Lemma~\ref{lemma:relativekHN} there exists a relative $\ell$-Harder-Narasimhan filtration of $q^*E$, which one can use to prove the following result:

\begin{lemma}\label{lemma:relativeHN}
For $a > 0$ there exists a dense open subset $U_a \subset \Pi_a$ and a quotient $q^*E \to \sheaf{F}_a$ such that
\begin{enumerate}[(1),nolistsep]
    \item each divisor $D \in U_a$ is smooth and integral,
    \item $\sheaf{F}_a$ is flat over $U_a$,
    \item for $[D] \in U_a$, $E|_D \to \sheaf{F}_a|_D$ is the minimal $\ell$-destabilizing quotient of $E|_D$,
    \item for $a \gg 0$, $E|_D$ is $(\ell-1)$-semistable.
\end{enumerate}
\end{lemma}
\begin{proof}
$(2)$ and $(3)$ are direct consequences of Lemma~\ref{lemma:relativekHN}. We may also assume by induction that $E|_D$ is $(\ell-1)$-semistable for general $[D] \in \Pi_a$ of sufficiently large degree. Finally $(1)$ holds by Bertini's Theorem.
\end{proof}

Since $\sheaf{F}_a$ is flat over $U_a$, the Hilbert polynomial of $\sheaf{F}_a|_D$ does not vary with $D \in U_a$. Hence we set $P(a) \coloneqq P(\sheaf{F}_a|_D)$,  $\rk(a)\coloneqq \rk(\sheaf{F}_a|_D)$, $\alpha_{n-j}(a) \coloneqq \alpha_{n-j}(\sheaf{F}_a|_D)$ and so forth.

\begin{remark}\label{rmk:betaCt}
Since $E|_D$ is already $(\ell-1)$-semistable, we have $p_{\ell-1}(a)=p_{\ell-1}(E|_D)$ for $a\gg 0$. Then
\begin{align*}
    \frac{\beta_{n-j}(a)}{a\rk(a)} = \frac{\beta_{n-j}(E|_D)}{a\rk(E|_D)} = \frac{\alpha_{n-j+1}(E)}{\rk(E)}
\end{align*}
is constant for any $1 \leq j \leq \ell$ and $a \gg 0$.
\end{remark}

Let $D = D_1 + D_2$ be a simple normal crossing divisor of degree $a = a_1 + a_2$ with $[D_1] \in U_{a_1}$ and $[D_2] \in U_{a_2}$. Then there exists a smooth (locally closed) curve $C \subset \Pi_a$ passing through $[D]$ such that $C \setminus [D] \subset U_a$ and $Z_C \coloneqq Z_{a} \times_{\Pi_a} C$ is smooth in codimension 2 (see \cite[Lem.~7.2.3]{huybrechts2010geometry}). Restricting $q^*E \to \sheaf{F}_a$ to $Z_{C \setminus [D]}$ gives a quotient $q^*E|_{Z_{C \setminus [D]}} \to \sheaf{F}_a|_{Z_{C \setminus [D]}}$, which induces a morphism of schemes
\begin{align*}
    C \setminus [D] \to \Quot_{Z_{a}/\Pi_a}(q^*E,P(a))
\end{align*}
As $C$ is a smooth curve and the Quot scheme is projective, we can extend the above morphism over $[D]$ and get a quotient $q^*E|_{Z_C} \to \sheaf{F}_C$. 

In what follows $i=1,2$. Let $F \coloneqq (\sheaf{F}_C|_D)/T(\sheaf{F}_C|_D)$, $F_i \coloneqq F|_{D_i}$ and denote by $T_i \coloneqq T(F_i)$ the torsion of $F_i$. For $1 \leq j \leq \ell + 1$, consider the following property $(P_j)$: 
\begin{enumerate}[leftmargin=*]
    \item for $a \gg 0$, $T(\sheaf{F}_C|_D)$ is supported in codimension $\geq j$ in $D$ and
    \begin{align*}
        \frac{\alpha_{n-j}(F)}{\rk(F)}  = \frac{\alpha_{n-j}(a)}{\rk(a)},
    \end{align*}
    \item for $a_i \gg 0$, $T_i$ is supported in codimension $\geq j$ in $D_i$ and
    \begin{align*}
        \frac{\alpha_{n-j}(F_i/T_i)}{\rk(F_i/T_i)}  = \frac{\alpha_{n-j}(a_i)}{\rk(a_i)}.
    \end{align*}
\end{enumerate}

We will show by induction that $(P_j)$ holds true for $1 \leq j \leq \ell + 1$. First note that $\alpha_{n-1}(F) = \alpha_{n-1}(\sheaf{F}_C|_D) = \alpha_{n-1}(a)$ by the flatness of $\sheaf{F}_C$. The second part of $(P_1)$ is also clearly satisfied. We next show the induction step, whose proof we divided in a sequence of lemmas. In what follows, let $j > 1$ and suppose that $(P_1), \ldots, (P_{j-1})$ hold true.

\begin{lemma}\label{lemma:alphaineq}
For $a_1, a_2 \gg 0$ we have
\begin{align}\label{eq:gammaineq}
    \frac{\alpha_{n-j}(a)}{\rk(a)} \geq {}& \frac{\alpha_{n-j}(F_1/T_1)}{\rk(F_1/T_1)} + \frac{\alpha_{n-j}(F_2/T_2)}{\rk(F_2/T_2)} - \notag\\
    {}& \sum_{l=1}^{j-1} \frac{(-1)^{l-1}}{2l!} \left(a_2^l\frac{\alpha_{n-j+l}(F_1/T_1)}{\rk(F_1/T_1)} + a_1^l\frac{\alpha_{n-j+l}(F_2/T_2)}{\rk(F_2/T_2)}\right).
\end{align}
\end{lemma}
\begin{proof}
Consider the short exact sequence
\begin{align*}
    0 \to \sheaf{O}_D \to \sheaf{O}_{D_1} \oplus \sheaf{O}_{D_2} \to \sheaf{O}_{D_1 \cap D_2} \to 0
\end{align*}
corresponding to the normal crossing divisor $D = D_1 + D_2$. After tensoring with the pure sheaf $F$ the sequence
\begin{align}\label{eq:Fexact}
    0 \to F \to F_1 \oplus F_2 \to F|_{D_1 \cap D_2} \to 0
\end{align}
remains exact. We get by additivity
\begin{align*}
    \alpha_{n-j}(F) = \alpha_{n-j}(F_1) + \alpha_{n-j}(F_2) - \alpha_{n-j}(F|_{D_1 \cap D_2}).
\end{align*}
From the short exact sequence
\begin{align}\label{eq:shortTorsion}
    0 \to T_i \to F_i \to F_i/T_i\to 0,
\end{align}
we also have
\begin{align*}
    \alpha_{n-j}(F_i) = \alpha_{n-j}(T_i) + \alpha_{n-j}(F_i/T_i).
\end{align*}
As $F_i/T_i$ has no torsion over $D_1 \cap D_2$, the restriction of \eqref{eq:shortTorsion} to $D_1 \cap D_2$ gives the short exact sequence
\begin{align*}
    0 \to T_i|_{D_1 \cap D_2} \to F|_{D_1 \cap D_2} \to (F_i/T_i)|_{D_1 \cap D_2} \to 0.
\end{align*}
Thus
\begin{align*}
    \alpha_{n-j}(F|_{D_1 \cap D_2}) = \alpha_{n-j}(T_i|_{D_1 \cap D_2}) + \alpha_{n-j}((F_i/T_i)|_{D_1 \cap D_2}).
\end{align*}
Putting together the above identities and using $\alpha_{n-j}(T_i) \geq \alpha_{n-j}(T_i|_{D_1 \cap D_2})$, we get
\begin{align}\label{eq:geq2}
    \alpha_{n-j}(F) \geq{}& \alpha_{n-j}(F_1) + \alpha_{n-j}(F_2/T_2) - \alpha_{n-j}((F_2/T_2)|_{D_1 \cap D_2})\notag\\
    \geq{}& \alpha_{n-j}(F_1/T_1) + \alpha_{n-j}(F_2/T_2) -\alpha_{n-j}((F_2/T_2)|_{D_1 \cap D_2}),
\end{align}
where the last inequality holds since the torsion is supported in codimension $\geq j-1$ by the induction hypothesis.

On the other hand, from the short exact sequence
\begin{align*}
    0 \to (F_2/T_2)(-D_1) \to F_2/T_2 \to (F_2/T_2)|_{D_1 \cap D_2} \to 0,
\end{align*}
it follows by formula \eqref{eq:alphaComputations} that
\begin{align*}
    \alpha_{n-j}((F_2/T_2)|_{D_1 \cap D_2}) = \sum_{l=1}^{j-1} \frac{(-1)^{l-1}a_1^l}{l!}\alpha_{n-j+l}(F_2/T_2).
\end{align*}
From this we can rewrite \eqref{eq:geq2} as
\begin{align*}
    \alpha_{n-j}(F) \geq \alpha_{n-j}(F_1/T_1) + \alpha_{n-j}(F_2/T_2) - \sum_{l=1}^{j-1}  \frac{(-1)^{l-1}a_1^l}{l!}\alpha_{n-j+l}(F_2/T_2).
\end{align*}
By a similar computation, we also have
\begin{align*}
    \alpha_{n-j}(F) \geq \alpha_{n-j}(F_1/T_1) + \alpha_{n-j}(F_2/T_2) - \sum_{l=1}^{j-1}  \frac{(-1)^{l-1}a_2^l}{l!}\alpha_{n-j+l}(F_1/T_1).
\end{align*}
Adding the last two inequalities and dividing by 2 we obtain
\begin{align}\label{eq:linearcombination}
    \alpha_{n-j}(F) \geq {}&\alpha_{n-j}(F_1/T_1) + \alpha_{n-j}(F_2/T_2) - \notag \\
    {}& \sum_{l=1}^{j-1} \frac{(-1)^{l-1}}{2l!} (a_2^l\alpha_{n-j+l}(F_1/T_1) + a_1^l\alpha_{n-j+l}(F_2/T_2)).
\end{align}

By the flatness of $\sheaf{F}_C$ over $C$, we have $\alpha_{n-j}(\sheaf{F}_C|_D) = \alpha_{n-j}(a)$. Moreover, since $\Supp(T(\sheaf{F}_C|_D))$ has codimension $\geq j-1$ in $D$ by $(P_{j - 1})$, we get
\begin{align}\label{eq:betaF}
    \alpha_{n-j}(a) = \alpha_{n-j}(\sheaf{F}_C|_D) \geq \alpha_{n-j}(F).
\end{align}
We also have $\rk(F_i/T_i)=\rk(F)=\rk(a)$ (see \cite[Lem.~2.12]{gangopadhyaymehta}). From these considerations, dividing \eqref{eq:linearcombination} by $\rk(a)$ we finally deduce \eqref{eq:gammaineq}. 
\end{proof}

\begin{lemma}\label{lemma:ineqbetweengamma}
For $a_1, a_2 \gg 0$ we have 
\begin{align*}
    \frac{\beta_{n-j}(a)}{\rk(a)} \geq \frac{\beta_{n-j}(a_1)}{\rk(a_1)} + \frac{\beta_{n-j}(a_2)}{\rk(a_2)}.
\end{align*}
\end{lemma}
\begin{proof}
The basic idea is to use \eqref{eq:bernoulliInv} in order to replace the $\alpha$-terms in \eqref{eq:gammaineq} by $\beta$-terms. This way the LHS of \eqref{eq:gammaineq} can be rewritten as
\begin{align*}
     \frac{\alpha_{n-j}(a)}{\rk(a)} = \frac{\beta_{n-j}(a)}{\rk(a)} + \sum_{l=1}^{j-1} \frac{(-1)^{l}}{(l+1)!}a^{l} \frac{\beta_{n-j+l}(a)}{\rk(a)}.
\end{align*}

We are now concerned with the RHS of \eqref{eq:gammaineq}. A straightforward computation gives
\begin{align*}
&\sum_{l=1}^{j-1} \frac{(-1)^{l-1}a_2^l}{2l!} \frac{\alpha_{n-j+l}(F_1/T_1)}{\rk(F_1/T_1)} =\\
    ={}&\sum_{l=1}^{j-1} \frac{(-1)^{l-1}a_2^l}{2l!}\sum_{s=1}^{j-l} \frac{(-1)^{s-1}a_1^{s-1}}{s!}\frac{\beta_{n-j+l+s-1}(a_1)}{\rk(a_1)} \tag{by \eqref{eq:bernoulliInv} and $(P_{<j}.2$)}\\ ={}&\sum_{p=1}^{j-1}\frac{\beta_{n-j+p}(a_1)}{\rk(a_1)}\sum_{l=1}^p \frac{(-1)^{l-1}a_2^l}{2l!}\frac{(-1)^{p-l} a_1^{p-l}}{(p-l+1)!}\tag{by the change of variables $p=s+l-1$}\\
    ={}& \sum_{p=1}^{j-1}\frac{\beta_{n-j+p}(a_1)}{a_1\rk(a_1)}\frac{(-1)^{p-1}}{2(p+1)!}((a_1 + a_2)^{p+1} - a_1^{p+1} - a_2^{p+1})\tag{by binomial formula}\\
    ={}& \sum_{l=1}^{j-1}\frac{(-1)^{l-1}}{2(l+1)!}\frac{\beta_{n-j+l}(a_1)}{a_1\rk(a_1)}(a^{l+1} - a_1^{l+1} - a_2^{l+1})\tag{by the change of variables $l=p$}.
\end{align*}
In a similar fashion,
\begin{align*}
    \sum_{l=1}^{j-1} \frac{(-1)^{l-1}a_1^l}{2l!} \frac{\alpha_{n-j+l}(F_2/T_2)}{\rk(F_2/T_2)} = \sum_{l=1}^{j-1}\frac{(-1)^{l-1}}{2(l+1)!}\frac{\beta_{n-j+l}(a_2)}{a_2\rk(a_2)}(a^{l+1} - a_2^{l+1}-a_1^{l+1}).
\end{align*}
Adding the last two identities, for $a_1, a_2 \gg 0$ we obtain by Remark~\ref{rmk:betaCt} that
\begin{align}\label{eq:dif}
     &\sum_{l=1}^{j-1} \frac{(-1)^{l-1}}{2l!} (a_2^l\frac{\alpha_{n-j+l}(F_1/T_1)}{\rk(F_1/T_1)} + a_1^l\frac{\alpha_{n-j+l}(F_2/T_2)}{\rk(F_2/T_2)}=\notag\\
    ={}& \sum_{l=1}^{j-1} \frac{(-1)^{l-1}}{(l+1)!}\left(\frac{\beta_{n-j+l}(a)}{a\rk(a)}a^{l+1} - \frac{\beta_{n-j+l}(a_1)}{a_1\rk(a_1)}a_1^{l+1}-\frac{\beta_{n-j+l}(a_2)}{a_2\rk(a_2)}a_2^{l+1}\right).
\end{align}

As $F_i/T_i$ is a torsion-free quotient of $E|_{D_i}$, it must satisfy
\begin{align}\label{eq:betaFi}
    \frac{\alpha_{n-j}(F_i/T_i)}{\rk(F_i/T_i)}\geq \frac{\alpha_{n-j}(\sheaf{F}_{a_i}|_{D_i})}{\rk(\sheaf{F}_{a_i}|_{D_i})} = \frac{\alpha_{n-j}(a_i)}{\rk(a_i)}.
\end{align}
Hence the RHS of $\eqref{eq:gammaineq}$ satisfies
\begin{align*}
{}&\frac{\alpha_{n-j}(F_1/T_1)}{\rk(F_1/T_1)} + \frac{\alpha_{n-j}(F_2/T_2)}{\rk(F_2/T_2)} -
    \sum_{l=1}^{j-1} \frac{(-1)^{l-1}}{2l!} \left(a_2^l\frac{\alpha_{n-j+l}(F_1/T_1)}{\rk(F_1/T_1)} + a_1^l\frac{\alpha_{n-j+l}(F_2/T_2)}{\rk(F_2/T_2)}\right)
 \\
    \geq{}&\frac{\alpha_{n-j}(a_1)}{\rk(a_1)} + \frac{\alpha_{n-j}(a_2)}{\rk(a_2)} - \tag{by \eqref{eq:dif} and \eqref{eq:betaFi}}\\
    {}& \sum_{l=1}^{j-1} \frac{(-1)^{l-1}}{(l+1)!}\left(\frac{\beta_{n-j+l}(a)}{a\rk(a)}a^{l+1} - \frac{\beta_{n-j+l}(a_1)}{a_1\rk(a_1)}a_1^{l+1}-\frac{\beta_{n-j+l}(a_2)}{a_2\rk(a_2)}a_2^{l+1}\right)\\
    \geq{}&\frac{\beta_{n-j}(a_1)}{\rk(a_1)} + \frac{\beta_{n-j}(a_2)}{\rk(a_2)}+ \sum_{l=1}^{j-1} \frac{(-1)^{l}}{(l+1)!}a^{l} \frac{\beta_{n-j+l}(a)}{\rk(a)}. \tag{by \eqref{eq:bernoulliInv}}
\end{align*}
Comparing both sides of \eqref{eq:gammaineq}, we conclude that for $a_1, a_2 \gg 0$
\begin{align*}
    \frac{\beta_{n-j}(a)}{\rk(a)} \geq \frac{\beta_{n-j}(a_1)}{\rk(a_1)} + \frac{\beta_{n-j}(a_2)}{\rk(a_2)}.
\end{align*}
\end{proof}

\begin{lemma}\label{lemma:gammaConstant}
$\frac{\beta_{n-j}(a)}{a\rk(a)}$ is constant for $a \gg 0$.
\end{lemma}
\begin{proof}
We already know the result holds true for $1 \leq j \leq \ell$ by Remark~\ref{rmk:betaCt}, so it only remains to prove it when $j = \ell + 1$.

First, we show that $\{ \frac{\beta_{n-1-\ell}(a)}{a\rk(a)} : a > 0\}$ is a discrete subset of $\Q$. Recall that $q: Z_{a} \to X$ is isomorphic to the projective bundle $\CP(\sheaf{K}^\vee) \to X$ and, via this isomorphism, we have $\sheaf{O}_{\CP(\sheaf{K}^\vee)}(1) \cong p^*(\sheaf{O}_{\Pi_a}(1))$ (see Remark~\ref{rmk:PicDec}). Denote by $\xi$ the class of $\sheaf{O}_{\Pi_a}(1)$ in the Grothendieck group $K(\Pi_a)$ and set $r \coloneqq \rk(\sheaf{K})$. We have the following isomorphism of $K(X)$-algebras (see \cite[Thm.~4.5]{manin69})
\begin{align*}
    K(Z_{a}) \cong K(\CP(\sheaf{K}^\vee)) \cong K(X)[T]/(\sum_{i=0}^{r} (-1)^i cl(\Lambda^{r-i}(\sheaf{K}^\vee))T^i),
\end{align*}
where $cl(-)$ denotes the corresponding Grothendieck class. Moreover 
\begin{align*}
    T \text{ mod }(\sum_{i=0}^{r} (-1)^i cl(\Lambda^{r-i}(\sheaf{K}^\vee))T^i) = cl(p^*\xi) \in K(Z_{a}).
\end{align*}
Using this isomorphism we can write
 \begin{align}\label{eq:grothendieckclasses}
     cl(\sheaf{F}_a) &= \sum_{i=0}^{r-1} cl(G_i)cl(p^*\xi)^i
 \end{align}
for $G_i$ coherent sheaves on $X$. As $X$ is smooth and projective, any $G_i$ has a finite free resolution. Thus, there is a finite sum
\begin{align*}
    cl(G_i) = \sum_p n_{i,p} cl(G_{i,p})
\end{align*}
in $K(X)$ with $n_{i,p} \in \Z$ and $G_{i,p}$ locally free sheaves on $X$.
Restricting \eqref{eq:grothendieckclasses} to a fiber over $[D] \in U_a$, we obtain 
\begin{align*}
    i^*cl(\sheaf{F}_a) = \sum_{i=0}^{r-1} \sum_p n_{i,p} cl(G_{i,p}|_D)
\end{align*}
in $K(Z_{[D]}) \cong K(D)$, where $i:Z_{[D]}\to Z$ is the natural immersion. For $1 \leq j \leq \ell + 1$, we get by additivity
\begin{align*}
    \frac{\alpha_{n-j}(a)}{a\rk(a)} &= \sum_{i=0}^{r-1} \sum_p n_{i,p} \frac{\alpha_{n-j}(G_{i,p}|_D)}{a\rk(a)}\\
    &= \sum_{i=0}^{r-1} \sum_p n_{i,p} \sum_{l=1}^{j}  \frac{(-1)^{l-1}a^{l-1}}{l!}\frac{\alpha_{n-j+l}(G_{i,p})}{\rk(a)} \in \frac{\Z}{(\ell+1)!\rk(E)!}.
\end{align*}
from which we obtain that $\{ \frac{\beta_{n-1-\ell}(a)}{a\rk(a)} : a > 0\}$ is discrete in $\Q$.

Second, let us show that the function $a \mapsto \frac{\beta_{n-1-\ell}(a)}{a\rk(a)}$ is bounded from above. By Remark~\ref{rmk:betaCt} we have
\begin{align}\label{eq:betaE}
    \frac{\beta_{n-j}(a)}{a\rk(a)}= \frac{\alpha_{n-j+1}(E)}{\rk(E)}
\end{align}
for $1 \leq j \leq \ell$. Then
\begin{align*}
   \frac{\alpha_{n-1-\ell}(E|_D)}{a\rk(E|_D)}
   ={}& \sum_{l=1}^{\ell+1}  \frac{(-1)^{l-1}a^l}{l!}\frac{\alpha_{n-1-\ell+l}(E)}{a\rk(E|_D)} \tag{by \eqref{eq:alphaComputations}}\\
   ={}& \frac{\alpha_{n-\ell}(E)}{\rk(E|_D)} +  \sum_{l=2}^{\ell+1}  \frac{(-1)^{l-1}a^{l-1}}{l!}\frac{\beta_{n-1-\ell+l}(a)}{a\rk(a)} \tag{by \eqref{eq:betaE}}.
\end{align*}
On the other hand, as $E|_D \to \sheaf{F}_a|_D$ is an $\ell$-destabilizing quotient, we obtain
\begin{align*}
    \frac{\alpha_{n-1-\ell}(E|_D)}{a\rk(E|_D)} \geq{}& \frac{\alpha_{n-1-\ell}(a)}{a\rk(a)}\\ 
    \geq{}& \frac{\beta_{n-1-\ell}(a)}{a\rk(a)} + \sum_{l=2}^{\ell+1} \frac{(-1)^{l-1}a^{l-1}}{l!} \frac{\beta_{n-1-\ell+l}(a)}{a\rk(a)} \tag{by \eqref{eq:bernoulliInv}}.
\end{align*}
Then
\begin{align*}
    \frac{\beta_{n-1-\ell}(a)}{a\rk(a)} \leq \frac{\alpha_{n-\ell}(E)}{\rk(E)}.
\end{align*}

Finally, we prove that $\frac{\beta_{n-1-\ell}(a)}{a\rk(a)}$ is constant for $a \gg 0$. Since the map is bounded from above and discrete, its maximum is attained at an integer $b > 0$. Consider the set
\begin{align*}
    B = \{ \frac{\beta_{n-1-\ell}(b')}{b'\rk(b')} : b' \text{ and }b \text{ are relatively prime}\}.
\end{align*}
There exists $b' \in \N$ such that $\frac{\beta_{n-1-\ell}(b')}{b'\rk(b')}$ is the maximum of $B$. Suppose that $a = cb + c'b'$ for some integers $c,c' > 0$. By Lemma~\eqref{lemma:ineqbetweengamma} there is an integer $c_0 > 0$ such that for $c,c' \geq c_0$
\begin{align*}
    \frac{\beta_{n-1-\ell}(a)}{a\rk(a)} \geq{}&  \frac{c}{a}\frac{\beta_{n-1-\ell}(b)}{\rk(b)} + \frac{c'}{a}\frac{\beta_{n-1-\ell}(b')}{\rk(b')}\\
    \geq{}&\frac{cb}{a}\frac{\beta_{n-1-\ell}(b)}{b\rk(b)} + \frac{c'b'}{a}\frac{\beta_{n-1-\ell}(b)}{b'\rk(b')}\\
    \geq{}& \frac{\beta_{n-1-\ell}(b')}{b'\rk(b')}.
\end{align*}
But $a$ and $b'$ are coprime, hence we have equality above. In particular
\begin{align*}
    \frac{\beta_{n-1-\ell}(b)}{b\rk(b)} = \frac{\beta_{n-1-\ell}(b')}{b'\rk(b')},
\end{align*}
which further gives
\begin{align*}
    \frac{\beta_{n-1-\ell}(a)}{a\rk(a)} = \frac{\beta_{n-1-\ell}(b)}{b\rk(b)}.
\end{align*}
As $b$ and $b'$ are coprime, for $a \gg 0$ there are always integers $c, c' \geq c_0$ such that $a = cb + c'b'$, which concludes the proof. 
\end{proof}

\begin{cor}
$(P_{j})$ holds true.
\end{cor}
\begin{proof}
By Lemma~\ref{lemma:gammaConstant} and Lemma~\ref{lemma:ineqbetweengamma} we get
\begin{align*}
    \frac{\beta_{n-j}(a)}{\rk(a)} = \frac{\beta_{n-j}(a_1)}{\rk(a_1)} + \frac{\beta_{n-j}(a_2)}{\rk(a_2)}
\end{align*}
for $a_1,a_2 \gg 0$. It follows that all inequalities employed so far in the proof are in fact equalities, which proves $(P_j)$.
\end{proof}

\begin{cor}
$\rk(a)$ is constant for $a \gg 0$. 
\end{cor}
\begin{proof}
For $a_i \gg 0$, since $(P_j)$ is true for $1 \leq j \leq \ell + 1$, and so $p_{\ell}(F_i/T_i)= p_{\ell}(a_i) = p_{\ell}(\sheaf{F}_{a_i}|_{D_i})$. Moreover, as $\sheaf{F}_{a_i}|_{D_i}$ is the minimal $\ell$-destabilizing quotient of $E|_{D_i}$, we have $\rk(F_i/T_i) = \rk(a) \leq \rk(a_i)$. Then there is an integer $a_0 > 0$ such that $\rk(a) \leq \min\{\rk(a_1),\rk(a_2)\}$ for all $a_1, a_2 \geq a_0$. Let $R \coloneqq \{ \rk(a) : a \geq a_0 \}$. Since $R$ is discrete and bounded, there exists $b \geq a_0$ such that $\rk(b)$ is the minimum of $R$. Then, for $a \geq b + a_0$ we have $\rk(a) \leq \min\{\rk(b), \rk(a-b)\} = \rk(b)$. This implies that $\rk(a) = \rk(b)$ and consequently, $\rk(a)$ is constant for $a \gg 0$.
\end{proof}

Denote by $\sheaf{L}_a \coloneqq \Lambda^{r(a)}(\sheaf{F}_a)^{\vee\vee}$ the determinant of $\sheaf{F}_a$. By Remark~\ref{rmk:PicDec} $\Pic(Z_{a}) \cong \Pic(X) \times \Pic(\Pi_a)$, thus $\sheaf{L}_a \cong p^*M_a \otimes q^*L_a$ for some $L_a \in \Pic(X)$ and $M_a \in \Pic(\Pi_a)$. In particular $\det(\sheaf{F}_a|_D) \cong L_a|_D$ for $[D] \in U_a$.

\begin{lemma}\label{lemma:sameLineB}
There is a line bundle $L \in \Pic(X)$ such that $L \cong L_a$ for $a \gg 0$.
\end{lemma}

\begin{proof}
We have proved that $p_{\ell}(F_i/T_i)= p_{\ell}(a_i) = p_{\ell}(\sheaf{F}_{a_i}|_{D_i})$. Then $F_i/T_i \cong \sheaf{F}_{a_i}|_{D_i}$ since $\sheaf{F}_{a_i}|_{D_i}$ is the (unique) minimal $\ell$-destabilizing quotient of $E|_{D_i}$ and has the same rank as $F_i/T_i$. In particular they have the same determinant line bundle. This is enough to prove the result, which now follows exactly the same as in \cite[Lem.~7.2.7]{huybrechts2010geometry}. 
\end{proof}

By using Lemma~\ref{lemma:sameLineB}, one can extend the minimal destabilizing quotient of $E|_D$ over $X$. Precisely, we prove the following result:

\begin{prop}\label{prop:extension}
There is a torsion-free quotient $E \to F$ on $X$ that extends $E|_D \to \sheaf{F}_a|_D$ for $[D] \in U_a$ and $a \gg 0$.
\end{prop}
\begin{proof}[Sketch of proof]
For a more detailed exposition of the argument below we refer the reader to \cite[p.~201]{huybrechts2010geometry}. Let $[D] \in U_a$ and consider the destabilizing quotient $E|_D \to \sheaf{F}_a|_D$. Clearly, there is an open subset $D' \subset D$ such that $E|_{D'}$ and $\sheaf{F}_a|_{D'}$ are both locally free sheaves and $\codim(D \setminus D',D) \geq 2$. If $r \coloneqq \rk(a)$, then by taking the $r$-th exterior power of $E|_{D} \to \sheaf{F}_a|_{D}$ we obtain a morphism $\sigma_D : \Lambda^r(E|_D) \to L|_D$, which is surjective over $D'$. Note that here we used the crucial fact that the determinant line bundles $L_a$ are isomorphic for $a \gg 0$ by Lemma~\ref{lemma:sameLineB}. Now, $\sigma_D$ induces a map
\begin{align*}
    D' \to \Grass(E,r) \subset \CP(\Lambda^r E),
\end{align*}
where the last inclusion is given by the Plucker embedding. By a standard argument using Serre duality, for $a \gg 0$ one can uniquely extend $\sigma_D$ to a morphism $\sigma: \Lambda^r E \to L$, which is surjective over an open subset $X' \subset X$ such that $\codim(X \setminus X',X) \geq 2$. This induces a map $X' \to \CP(\Lambda^r E)$ which actually factorizes through $\Grass(E,r)$. So we get a quotient $E|_{X'} \to F'$ that extends $E|_{D'} \to \sheaf{F}_a|_{D'}$. Let $G'$ be its kernel and take $G \coloneqq \iota^{-1}(i_*G')$, where $i: X' \to X$ is the inclusion and $\iota: E \to E^{\vee \vee}$ is the canonical injection of $E$ into its double dual. Then we have a short exact sequence
\begin{align*}
    0 \to G \to E \to F \to 0
\end{align*}
on $X$ such that $F|_{X'} = F'$. Restricting it to $D$, we get a short exact sequence
\begin{align*}
     0 \to G|_D \to E|_D \to F|_D \to 0.
\end{align*}
On the other side, if $G_a$ denotes the kernel of $E|_{D} \to \sheaf{F}_a|_{D}$, then we have another short exact sequence
\begin{align*}
    0 \to G_a \to E|_D \to \sheaf{F}_a|_{D} \to 0.
\end{align*}
Note that, by construction, $G|_{D'}$ is isomorphic to $G_a|_{D'}$. Since both $G|_D$ and $G_a$ are saturated in $E|_D$ and have the same double dual, it follows by the lemma below that they are in fact isomorphic over $D$. Consequently, $F|_D$ is isomorphic to $\sheaf{F}_a|_D$.
\end{proof}

\begin{lemma}
Let $M$ be a torsion-free sheaf and $N \subset M$ a subsheaf. Let $\iota : M \to M^{\vee \vee}$ be the canonical injection. Then the saturation of $N$ in $M$ is equal to $ \iota^{-1}(N^{\vee \vee})$.
\end{lemma}
\begin{proof}
Denote by $S$ the saturation of $N$ in $M$ and set $N' \coloneqq \iota^{-1}(N^{\vee \vee})$. Then we have an injection $M/N' \to M^{\vee\vee}/N^{\vee\vee}$ such that $M^{\vee\vee}/N^{\vee\vee}$ is torsion-free, and so $M/N'$ is also torsion-free. Hence $S \subset N'$. We get a diagram as follows
\begin{align*}
    \xymatrix{0\ar[r] &  S \ar[r] \ar[d]^\alpha &  N' \ar[r] \ar[d]^\beta &  N'/S \ar[r] \ar[d] &  0 \\ 
             0 \ar[r] &  S^{\vee \vee} \ar[r] &  N'^{\vee \vee} \ar[r] &  0 \ar[r] &0}
\end{align*}
where $\alpha$ and $\beta$ are the canonical maps. By applying Snake's Lemma, we obtain an injection $N'/S \to \Coker(\alpha)$. But $N'/S \subset E/S$ is torsion-free and $\Coker(\alpha)$ is supported in positive codimension, thus $N'/S = 0$.
\end{proof}

We can complete now the proof of Theorem~\ref{thm:GiesekerRestriction}. Suppose the theorem were false. Then $E|_D$ is not $\ell$-semistable and admits an $\ell$-destabilizing quotient $E|_D \to \sheaf{F}_a|_D$ for $D \in \Pi_a$ general and $a \gg 0$. By Proposition~\ref{prop:extension} we can extend this quotient over $X$ for $D \in U_a$ of sufficiently large degree, which yields a contradiction by Lemma~\ref{cor:destabilizing}. 

\begin{cor}\label{cor:Giesekerss}
If $E$ is a Gieseker-semistable (i.e.~$n$-semistable) sheaf on $X$, then the restriction $E|_D$ is still Gieseker-semistable for a general divisor $D \in \Pi_a$ of sufficiently large degree.
\end{cor}

Clearly the result of Corollary~\ref{cor:Giesekerss} does not hold in general for Gieseker-stability, as we have already pointed out in the introduction.

\subsection*{The case of $\ell$-stability}
We next show Theorem~\ref{thm:restrictionStable}. Fix $\ell < n$ and let $E$ be an $\ell$-stable torsion-free sheaf on $X$. Recall that any $\ell$-stable sheaf is in particular $\ell$-simple. Therefore a necessary condition for the theorem to hold is that the restriction $E|_D$ to a general divisor of sufficiently large degree remains $\ell$-simple. We show below that this is indeed the case whenever $\ell < n$.

\begin{lemma}\label{lemma:restrictionSimple}
Under the same assumptions as above, the restriction $E|_D$ to a general divisor $D \in \Pi_a$ remains $\ell$-simple and torsion-free for $a \gg 0$.
\end{lemma}
\begin{proof}
By Lemma~\ref{lemma:simpleDepth} there exists a sheaf $G$ on $X$ which is isomorphic with $E$ in codimension $\ell$ and moreover satisfies \eqref{eq:depthProp}. In particular $G$ is also $\ell$-simple. In this case it is enough to show that the lemma holds for $G$.

As $\depth_x(G) \geq 2$ for any closed point $x \in X$, we have by Auslander-Buchsbaum formula
\begin{align*}
  \hd_x(E) = \dim(X) - \depth_x(E) \leq n-2,
\end{align*}
which further gives $\Exts^{n-1}(G,G)=0$.

Applying the functor $\Hom(G,-)$ to the short exact sequence
\begin{align*}
    0 \to G(-D) \to G \to G|_D \to 0,
\end{align*}
we obtain a long exact sequence
\begin{align}\label{eq:extG}
    \End(G) \to \End(G|_D) \to \Ext^{1}(G,G(-D)) \to \ldots.
\end{align}
By using Serre duality and the local-to-global Ext spectral sequence, one obtains
\begin{align*}
    \Ext^{1}(G,G(-D)) \cong H^0(X,\Exts^{n-1}(G,G) \otimes \omega_X(D)) = 0
\end{align*}
for $a \gg 0$, where $\omega_X$ denotes the dualizing line bundle on $X$. As $G$ is $\ell$-simple, it is in particular simple. It follows by \eqref{eq:extG} that $G|_D$ is also simple for $a \gg 0$. Moreover, for $[D] \in \Pi_a$ general we may further assume by \cite[Lem.~1.1.13]{huybrechts2010geometry} that $G|_D$ is a torsion-free sheaf satisfying \eqref{eq:depthProp}.

Let $Y \subset D$ be a closed subset such that  $\codim(Y,D) \geq \ell + 1$. As $G|_D$ satisfies \eqref{eq:depthProp}, we have by Lemma~\ref{lemma:simpleDepth} that
\begin{align*}
  \Hom(G|_D,G|_D) \to \Hom(G|_{D \setminus Y},G|_{D \setminus Y}) 
\end{align*}
is an isomorphism, which proves that $G|_D$ is also $\ell$-simple.
\end{proof}

We proceed with the proof of Theorem~\ref{thm:restrictionStable}. One main difference compared to the case of $\ell$-semistability lies in the fact that the Harder-Narasimhan filtration of an $\ell$-stable sheaf is trivial. We will use instead the so-called extended $\ell$-socle (see Definition~\ref{def:extendedSocle}) to construct a flat family of destabilizing quotients. This approach is due to Mehta and Ramanathan \cite{mehta1984restriction}. 

In what follows we suppose that the theorem is false, i.e. for any integer $a_0 > 0$ there exists $a \geq a_0$ such that $E|_D$ is $\ell$-unstable for general $D \in \Pi_a$. 

\begin{lemma}\label{lemma:flatFam}
Let $a > 0$ be an integer such that $E|_D$ is $\ell$-unstable and $\ell$-simple for general $[D] \in \Pi_a$. Then there exists a dense open subset $W_a \subset \Pi_a$ and a quotient $q^*E \to \sheaf{G}_a$ such that
\begin{enumerate}[(1),nolistsep]
  \item each divisor $D \in W_a$ is smooth and integral,
  \item $\sheaf{G}_a$ is flat over $W_a$,
  \item for $[D] \in W_a$, $E|_D \to \sheaf{G}_a|_D$ is $\ell$-destabilizing, i.e. $p_\ell(E|_D)=p_\ell(\sheaf{G}_a|_D)$,
  \item for $a \gg 0$, $E|_D$ is $\ell$-semistable.
\end{enumerate}
\end{lemma}
\begin{proof}
Since $k$ is algebraically closed, $E|_D$ is geometrically $\ell$-unstable. Recall that the property of being geometrically $\ell$-stable is open in flat families, cf.~Lemma~\ref{lemma:openness}. Thus, if $D_\xi \in \Pi_a$ denotes the generic point of $\Pi_a$, then $E|_{D_\xi} \coloneqq q^*E_{D_\xi}$ must be geometrically $\ell$-unstable on $Z_{D_\xi}$. On the other side, from Lemma~\ref{lemma:openSimple} we obtain that $E|_{D_\xi}$ is also $\ell$-simple. Hence, according to Lemma~\ref{lem:geomSt}, $E|_{D_\xi}$ is $\ell$-unstable.

Now consider the quotient $G_\xi$ of $E|_{D_\xi}$ by its $\ell$-extended socle (see Definition~\ref{def:extendedSocle}). Extend $G_\xi$ to a quotient $q^*E \to \sheaf{G}_a$ over $Z_a$. By generic flatness there is an open subset $W_a \subset \Pi_a$ such that $\sheaf{G}_a$ is flat over $W_a$ and $E|_D \to \sheaf{G}_a|_D$ is $\ell$-destabilizing for all $[D] \in W_a$. Moreover, by Bertini's Theorem and shrinking $W_a$ if necessary, we may assume that each divisor $D \in W_a$ is smooth and integral. Also note that (4) holds true by Theorem~\ref{thm:GiesekerRestriction}.
\end{proof}

\begin{lemma}
For $a \gg 0$ the restriction $E|_D$ to a general divisor $D \in \Pi_a$ is $\ell$-unstable.
\end{lemma}
\begin{proof}[Sketch of proof]
We argue by contradiction. Let $a_1 > 0$ be an integer such that $E|_{D_1}$ is $\ell$-stable for every divisor $D_1$ in some dense open subset $W_{a_1} \subset \Pi_{a_1}$. There exists $a \geq a_1$ such that $E|_D$ is $\ell$-unstable and $\ell$-simple (cf.~Lemma~\ref{lemma:restrictionSimple}) for general $[D] \in \Pi_a$. Then we have a flat quotient $q^*E \to \sheaf{G}_a$ over $W_a \subset \Pi_a$ as given by Lemma~\ref{lemma:flatFam}.

Set $a_2 \coloneqq a - a_1$. By Theorem~\ref{thm:GiesekerRestriction}, there exists a non-empty open subset $W_{a_2} \subset \Pi_{a_2}$ such that $E|_{D_2}$ is $\ell$-semistable for every $D_2 \in W_{a_2}$. Consider a normal crossing divisor $D = D_1 + D_2$ with $D_1 \in W_{a_1}$ and $D_2 \in W_{a_2}$. Similarly to the discussion below Remark~\ref{rmk:betaCt}, there exists a smooth curve $C \subset \Pi_a$ passing through $[D]$ such that $C \setminus [D] \subset W_a$, and there is a flat quotient $q^*E \to \sheaf{G}_C$ over $Z_C$ obtained by restricting $q^*E \to \sheaf{G}_a$. 

Set $G \coloneqq (\sheaf{G}_C|_D)/T(\sheaf{G}_C|_D)$ and $G_i \coloneqq G|_{D_i}$ for $i=1,2$. As $E|_{D_i}$ is $\ell$-semistable by assumption, we have
\begin{align*}
  p_\ell(G_i/T(G_i)) \geq p_\ell(E|_{D_i}).
\end{align*}
Then, by similar arguments to that used in the $\ell$-semistable case, one shows that in fact we have equality above. In particular
\begin{align*}
  p_\ell(G_1/T(G_1)) = p_\ell(E|_{D_1}),
\end{align*}
which further implies that $G_1/T(G_1) \cong E|_{D_1}$ as $E|_{D_1}$ is $\ell$-stable. However, we also have $\rk(G_1/T(G_1))= \rk(\sheaf{G}_a) < \rk(E)$ since $\sheaf{G}_a$ is a family of (proper) destabilizing quotients. This yields a contradiction. 
\end{proof}

We have shown above how to construct a flat family $q^*E \to \sheaf{G}_a$ of (proper) destabilizing quotients for every $a \gg 0$. As before, there are line bundles $L_a$ on $X$ such that $\det(\sheaf{G}_a|_D) = L_a|_D$ for $[D] \in W_a$.

\begin{lemma}
There is an infinite subset $N \subset \N$ and a line bundle $L \in \Pic(X)$ such that $L \cong L_a$ for every $a \in N$. 
\end{lemma}
\begin{proof}
The result follows as in the slope-stable case. We refer the reader to \cite[Sect.~7]{huybrechts2010geometry} or \cite[Sect.~8]{gangopadhyaymehta} for a complete proof.
\end{proof}

Now a similar argument as done for Proposition~\ref{prop:extension} shows that, for a sufficiently large integer $a \in N$, $E|_D \to \sheaf{G}_a|_D$ extends to an $\ell$-destabilizing quotient $E \to G_a$ on $X$. This will give a contradiction and complete the proof of Theorem~\ref{thm:restrictionStable}. 

%% file: tex/ModuliSpace.tex
\section{Application: the moduli space of $(n-1)$-semistable sheaves}\label{sect:moduliSpace}

In this section we work over the field of complex numbers $\C$. Choose a numerical Grothendieck class $c \in \grothnum{X}$ of dimension $n \coloneqq \dim(X) > 1$, which will fix the topological type of sheaves. We will use the restriction theorems to construct a moduli space of $(n-1)$-semistable sheaves of class $c$ on $X$. Afterwards, we will describe precisely its geometric points. 

\subsection*{The construction}

Since the family of $(n-1)$-semistable sheaves of class $c$ on $X$ is bounded, cf. Lemma~\ref{lemma:boundedness}, for sufficiently large $m > 0$ every such semistable sheaf $E$ is $m$-regular (in the sense of Castelnuovo--Mumford) and has Hilbert polynomial $P \coloneqq P(c)$. In particular $E(m)$ is globally generated with $P(m) = h^0(E(m))$. Thus the evaluation map induces a quotient
\begin{align*}
  H^0(E(m)) \otimes \sheaf{O}_X(-m) \to E \to 0.
\end{align*}
Let $V \coloneqq \C^{P(m)}$ and consider the subset $\Rss \subset \Quot(V \otimes \sheaf{O}_X(-m),P)$ of quotients $[q: V \otimes \sheaf{O}_X(-m) \to E]$ on $X$ such that
\begin{enumerate}[nosep]
  \item $E$ is $(n-1)$-semistable of class $c$,
  \item $q$ induces an isomorphism $V \to H^0(E(m))$.
\end{enumerate}
Then $\Rss$ is a locally closed subscheme that parametrizes $(n-1)$-semistable sheaves of class $c$ on $X$. However, note that any semistable sheaf is uniquely determined in $\Rss$ up to a base change of $V$. In this respect, $\Rss$ is endowed with a natural $G \coloneqq \SL{V}$ action. 

For the following result, see Definition~\ref{def:Gproper} regarding $G$-properness.

\begin{lemma}
$\Rss$ is a $G$-proper scheme over $\C$. 
\end{lemma}
\begin{proof}
Let $A$ be a discrete valuation ring over $\C$ of quotient field $K$. Consider a commutative diagram as follows
\begin{center}
\begin{tikzcd}[column sep=normal]
\Spec(K) \ar[d,"j"] \ar[r,"g"] &  \Rss \ar[d] \\  \Spec(A) \ar[r] & \Spec(\C)
\end{tikzcd}
\end{center}
The map $\Spec(K) \to \Rss$ corresponds to a quotient $[q_K: V \otimes K \to E_K] \in \Rss$, with $E_K$ $\ell$-semistable. By Lemma~\ref{lemma:properness}, we can extend $E_K $ to an $A$-flat family $E$ of $\ell$-semistable sheaves on $X$. This induces a quotient $[q_A : V \otimes A \to E] \in \Rss$, or equivalently an $A$-point $i : \Spec(A) \to \Rss$. As $E \otimes_A K \cong E_K$, $g$ and $i \circ j$ differ by a group element in $G(K)$, which completes the proof. 
\end{proof}

Next we want to define a $G$-semiample line bundle over $\Rss$, i.e. a $G$-equivariant line bundle globally generated by $G$-invariant sections. Let $\sheaf{F}$ denote the universal family of quotients over $\Rss$ with its natural $G$-action induced from $\Rss$. Choose a smooth divisor $D \in \Pi_a$. Since $\sheaf{F}_s(-D) \to \sheaf{F}_s$ is injective for every $s \in \Rss$, we have a short exact sequence
\begin{align}\label{eq:exactSeqFamily}
  0 \to \sheaf{F}(-D) \to \sheaf{F} \to \sheaf{F}|_{\Rss \times D} \to 0
\end{align}
such that $\sheaf{F}|_{\Rss \times D}$ remains $\Rss$-flat, cf. \cite[Thm.~22.5]{matsumura1989commutative}. Furthermore, for $m'$ sufficiently large we may also assume that each fiber of $\sheaf{F}|_{\Rss \times D}$ over $\Rss$ is $m'$-regular. In particular $R^i{p_1}_*(\sheaf{F}|_{\Rss\times D}(m')) = 0$ for $i > 0$ and ${p_1}_*(\sheaf{F}|_{\Rss\times D}(m'))$ is a locally free $G$-equivariant sheaf on $\Rss$ of rank $P'(m') \coloneqq P(c|_D(m'))$. 

Denote by $\mathcal{R}$ the projective frame bundle corresponding to ${p_1}_*(\sheaf{F}|_{\Rss\times D}(m'))$ and let $\pi : \mathcal{R} \to \Rss$ be the natural projection. Then there exists a natural quotient
\begin{align}\label{eq:frameQuotient}
 \sheaf{O}_{\mathcal{R}} \otimes \sheaf{O}_D(-m')^{P'(m')} \to \pi^*(\sheaf{F}|_{\Rss \times D}) \otimes \sheaf{O}_\pi(1) \to 0.
\end{align}
Set $V' \coloneqq \C^{P'(m')}$ and consider the Quot scheme $\Quot_D \coloneqq \Quot(V' \otimes \sheaf{O}_D(-m'),P')$ over $D$. Then the quotient map \eqref{eq:frameQuotient} induces a natural morphism $\Phi : \mathcal{R} \to \Quot_D$ and we have a diagram
\begin{center}
\begin{tikzcd}[column sep=normal]
\mathcal{R} \ar[d,"\pi"] \ar[r,"\Phi"] &  \Quot_D \ar[d,dashed,"\varphi'"]  \\ 
 \Rss \ar[r,dashed] & M_D
\end{tikzcd}
\end{center}
where $M_D$ denotes the Simpson moduli space of Gieseker-semistable sheaves of Hilbert polynomial $P'$ on $D$. Several remarks are in order:
\begin{remark}\label{remark:GITdivisor}
\begin{enumerate}[wide, labelwidth=!,nosep]
\item Let $\sheaf{F}'$ denote the universal family of quotients over $\Quot_D$. For  $l' > 0$, the Grothendieck class $[\sheaf{O}_D(l')] \in K(X)$ gives the following determinant line bundle
\begin{align*}
H_{l'} \coloneqq \lambda_{\sheaf{F}'}([\sheaf{O}_D(l')])
\end{align*}
over $\Quot_D$, see \eqref{eq:determinant}. In fact $H_{l'}$ is very ample for $l' \gg 0$, cf. \cite[Prop.~2.2.5]{huybrechts2010geometry}. Moreover, $\Quot_D$ is endowed with a natural $G' \coloneqq \SL{V'}$ action, which further induces a $G'$-linearization of $H_{l'}$.  Now consider the $G$-stable open subscheme $R_D \subset \Quot_D$ consisting of $m'$-regular Gieseker-semistable quotients $[q: V' \otimes \sheaf{O}_D(-m') \to F]$ such that the induced map $V' \to H^0(F(m'))$ is an isomorphism. By Simpson's construction \cite{simpson1994moduli}, for $l' \gg m' \gg 0$, $R_D$ is the locus of GIT-semistable points of the closure $\overline{R_D} \subset \Quot_D$ with respect to $H_{l'}$. Furthermore, the rational map $\varphi'$ is well-defined over $R_D$ and $R_D \to M_D$ is a good GIT quotient, in the sense of Mumford \cite{mumfordGIT}. 

\item Note that the image of $\Phi$ is contained in the open subscheme $Q \subset \Quot_D$ consisting of quotients $[q:V' \otimes \sheaf{O}_D(-m') \to F]$ such that 
\begin{enumerate}[(i),nosep]
 \item $H^i(F(m'))=0$ for all $i > 0$,
 \item the induced map $V' \to H^0(F(m'))$ is an isomorphism.
\end{enumerate}
Clearly $R_D$ is an open subset of $Q$. 
\item The given $G$-action on $\Rss$ lifts to a natural $G$-action on $\mathcal{R}$. There is also a natural $G'$-action on $\mathcal{R}$ induced by ${p_1}_*(\sheaf{F}|_{\Rss\times D}(m'))$ and which is compatible with its $G$-action. Then $\Phi$ is a $G \times G'$-equivariant morphism. 
\end{enumerate}
\end{remark}

Consider the following Grothendieck class in $K(X)$
\begin{align*}
  w \coloneqq \chi(c(m') \cdot h)[\sheaf{O}_X(l')]-\chi(c(l') \cdot h)[\sheaf{O}_X(m')],
\end{align*}
with $h \coloneqq [O_D] \in K(X)$, which we use to form the determinant line bundle
\begin{align*}
  \sheaf{L} \coloneqq \lambda_{\sheaf{F}}(w \cdot h)
\end{align*}
over $\Rss$, see \eqref{eq:determinant}. Note that $\sheaf{L}$ does not depend on the chosen divisor $D$ in $\Pi_a$, but only on the degree $a$. 

\begin{theorem}\label{thm:semiampleness}
Under the same assumptions as above, for $l' \gg m' \gg 0$ and $a \gg 0$, there exists an integer $\nu > 0$ such that $\sheaf{L}^\nu$ is globally generated by $G$-invariant sections over $\Rss$.  
\end{theorem}
\begin{proof}
Let $s \in \Rss$ be a closed point corresponding to an $(n-1)$-semistable quotient sheaf $E$ in $\Rss$. By Theorem~\ref{thm:GiesekerRestriction}, the restriction $E|_D$ to a general divisor $D \in \Pi_a$ is Gieseker-semistable for $a \gg 0$. Then, by Simpson's construction, there exists a $G'$-invariant section $\sigma \in H^0(\overline{R_D},H_{l'}^{\nu'})^{G'}$ non-vanishing at $[E|_D]$ for some $\nu' > 0$. By \cite[Lem.~4.1]{pavel2021moduli}, there exists a $G'$-equivariant isomorphism $H_{l'}^{\nu'} \cong \lambda_{\sheaf{F}'}(w|_D)^\nu$ over $Q$ for some $\nu > 0$. The restriction of $\sigma$ to $Q$ induces a $G'$-invariant section $\sigma_Q$ in $H^0(Q,\lambda_{\sheaf{F'}}(w|_D)^\nu)^{G'}$. Then its pull-back $\Phi^*(\sigma_Q)$ is a global $G \times G'$-invariant section, which descends to a $G$-invariant section in $H^0(\Rss,\lambda_{\sheaf{F}|_{\Rss \times D}}(w|_D)^\nu)^{G}$, as $\pi$ is a good quotient. Applying the determinant map $\lambda_{(-)}(w)$ to the exact sequence \eqref{eq:exactSeqFamily}, we obtain a $G$-equivariant isomorphism
\begin{align*}
  \sheaf{L} = \lambda_{\sheaf{F}}(w \cdot h) \cong \lambda_{\sheaf{F}|_{\Rss \times D}}(w|_D)
\end{align*}
over $\Rss$. From the above considerations, we have constructed a linear map
\begin{align*}
  \Gamma : H^0(\overline{R_D},H_{l'}^{\nu'})^{G'} \to H^0(\Rss, \sheaf{L}^\nu)^G
\end{align*}
such that $\Gamma(\sigma)$ is a $G$-invariant section of $\sheaf{L}^\nu$ non-vanishing at $s$. 
\end{proof}

We have shown above that $\Rss$ is a $G$-proper scheme over $\C$ endowed with a $G$-semiample line bundle $\sheaf{L}$. Since $G$ is reductive, one can form the Iitaka $G$-fibration corresponding to $(\Rss,\sheaf{L})$ as in \cite[Thm.~5.3]{pavel2021moduli}, which will define our moduli space $M$ of $(n-1)$-semistable sheaves of class $c$ on $X$. 

Consider the moduli functor
\begin{align*}
  \mf : \schC^\textnormal{op} \to (\textnormal{Sets})
\end{align*}
that associates to any scheme $S$ of finite type over $\C$ the set of all equivalence classes of $S$-flat families of $(n-1)$-semistable sheaves of class $c$ on $X$. Unfortunately our moduli space $M$ does not (co)represent $\mf$. Nevertheless, we have the following functorial property of $M$:

\begin{theorem}\label{thm:mainThm}
There exists a unique triple $(M, \sheaf{A}, e)$ formed of a projective scheme $M$ over $\C$ endowed with an ample line bundle $\sheaf{A}$ and a natural number $e > 0$ such that there is a natural transformation $\Psi: \mf \to \Hom(-,M)$, that associates to any $\C$-scheme of finite type $S$ and any $S$-flat family $\sheaf{E}$ of $(n-1)$-semistable sheaves of class $c$ on $X$ a classifying morphism $\Psi_\sheaf{E}: S \to M$, satisfying the following properties:
    \begin{enumerate}[(1)]
\item For any $S$-flat family $\sheaf{E}$ of $(n-1)$-semistable sheaves of class $c$ on $X$, the classifying morphism $\Psi_\sheaf{E}$ satisfies
\begin{align*}
    \Psi_\sheaf{E}^*(\sheaf{A}) \cong \lambda_\sheaf{E}(w \cdot h)^e,
\end{align*}
where $\lambda_\sheaf{E}(w \cdot h)$ is the determinant line bundle on $S$ defined by \eqref{eq:determinant}.
\item For any other triple $(M',\sheaf{A}',e')$, with $M'$ a projective scheme over $\C$, $\sheaf{A}'$ an ample line bundle on $M'$ and $e'$ a natural number satisfying property (1), we have $e | e'$ and there exists a unique morphism $\phi: M \to M'$ such that $\phi^*\sheaf{A}' \cong \sheaf{A}^{(e'/e)}$.
\end{enumerate}
\end{theorem}
\begin{proof}
The proof follows exactly as in \cite[Prop.~4.5]{greb2017compact} or \cite[Main Thm.]{pavel2021moduli}. We omit it.
\end{proof}

\subsection*{Geometric points of the moduli space}
We end this section by studying the geometric points of $M$. If $E$ is an $(n-1)$-semistable sheaf on $X$, then there exists a Jordan-H\"older filtration
\begin{align*}
  0 = E_0 \subset E_1 \subset \ldots \subset E_l = E
\end{align*}
such that its factors $E_i/E_{i-1}$ are $(n-1)$-stable with $p_{n-1}(E_i/E_{i-1}) = p_{n-1}(E)$. We denote by $\gr(E) \coloneqq \oplus_i E_i/E_{i-1}$ its corresponding graded torsion-free module. We say that $E$ is $(n-1)$-\textit{polystable} if $E$ is isomorphic to a direct sum of $(n-1)$-stable sheaves.

\begin{lemma}\label{lem:onlypolystable}
Let $E$ be an $(n-1)$-semistable sheaf on $X$ and $\gr(E)$ the graded module corresponding to a Jordan-H\"older filtration of $E$. Then $E$ and $\gr(E)$ give the same point in $M$.
\end{lemma}
\begin{proof}
It is not difficult to construct an $\A^1$-flat family $\sheaf{E}$ of $(n-1)$-semistable sheaves on $X$ such that $\sheaf{E}_0 \cong \gr(E)$ and $\sheaf{E}_t \cong E$ for all $t \neq 0$. Then the classifying morphism $\Psi_\sheaf{E} : \A^1 \to M$ corresponding to $\sheaf{E}$ is constant, which shows that $E$ and $\gr(E)$ define the same point in $M$. 
\end{proof}

Thus, it will suffice to analyze the separation of polystable sheaves. If $E$ is $(n-1)$-polystable, we denote by $E^{[n-1]}$ the corresponding torsion-free sheaf satisfying property $(\star)$ given by Lemma~\ref{lemma:simpleDepth}. Note that $E^{[n-1]}$ is also $(n-1)$-polystable as it differs from $E$ only in a finite number of points. We let $l_E$ denote the zero-cycle associated to $E^{[n-1]}/E$.

Next we aim to prove the following result:

\begin{theorem}\label{thm:geometricSep}
Let $E_1$ and $E_2$ be two $(n-1)$-polystable sheaves of class $c$ on $X$. Then $E_1$ and $E_2$ define the same point in $M$ if and only if $E_1^{[n-1]} \cong E_2^{[n-1]}$ and $l_{E_1} = l_{E_2}$.
\end{theorem}

\begin{remark}
Let $E$ be an $(n-1)$-semistable sheaf on $X$ and $\gr(E)$ the graded module corresponding to a Jordan-H\"older filtration of $E$. Note that $\gr(E)$ is uniquely determined only in codimension $n-1$. However, one can use the result above to show that the pair $(\gr(E)^{[n-1]},l_{\gr(E)})$ is in fact uniquely determined by $E$. Indeed, suppose we have two graded modules $\gr_1(E)$ and $\gr_2(E)$ corresponding to two different Jordan-H\"older filtrations of $E$. Then, by Lemma~\ref{lem:onlypolystable} and Theorem~\ref{thm:geometricSep}, we must have $\gr_1(E)^{[n-1]} \cong \gr_2(E)^{[n-1]}$ and $l_{\gr_1(E)} = l_{\gr_2(E)}$.
\end{remark}

We divided the proof of Theorem~\ref{thm:geometricSep} in a sequence of lemmas.

\begin{lemma}
Let $E_1$ and $E_2$ be two $(n-1)$-polystable sheaves of class $c$ on $X$ such that $E_1^{[n-1]} \ncong E_2^{[n-1]}$. Then $E_1$ and $E_2$ give distinct points in $M$.
\end{lemma}
\begin{proof}
By Theorem~\ref{thm:restrictionStable} and \cite[Lem.~1.1.13]{huybrechts2010geometry}, we can choose a general divisor $D \in \Pi_a$, with $a \gg 0$, such that
\begin{enumerate}[nosep]
  \item $E_1|_D$ and $E_2|_D$ remain polystable,
  \item $E_1|_D = E_1^{[n-1]}|_D$ and $E_2|_D = E_2^{[n-1]}|_D$ are torsion-free and satisfy $(\star)$.
\end{enumerate}
As in the proof of Lemma~\ref{lemma:restrictionSimple}, one can show that
\begin{align*}
  \Hom(E_1^{[n-1]}, E_2^{[n-1]}) \to \Hom(E_1^{[n-1]}|_D,E_2^{[n-1]}|_D)
\end{align*}
is an isomorphism for $a \gg 0$. By assumption $E_1^{[n-1]} \ncong E_2^{[n-1]}$, and thus $E_1|_D$ and $E_2|_D$ are non-isomorphic polystable sheaves on $D$. Then $E_1|_D$ and $E_2|_D$ give distinct points in the Simpson moduli space $M_D$ of Gieseker-semistable sheaves on $D$. Hence there exists a section $\sigma \in H^0(\overline{R}_D,H_{l'}^{\nu'})^{G'}$, for some $\nu' > 0$, that separates $E_1|_D$ and $E_2|_D$ (see Remark~\ref{remark:GITdivisor} for notation). Following the proof of Theorem~\ref{thm:semiampleness}, we find a section $\Gamma(\sigma) \in H^0(\Rss, \sheaf{L}^\nu)^G$, for $\nu > 0$, that separates $E_1$ and $E_2$ in $\Rss$. This implies that $E_1$ and $E_2$ give distinct points in $M$.
\end{proof}

\begin{lemma}
Let $E_1$ and $E_2$ be two $(n-1)$-polystable sheaves of class $c$ on $X$ such that $l_{E_1} \neq l_{E_2}$. Then $E_1$ and $E_2$ give distinct points in $M$.
\end{lemma}
\begin{proof}
The proof goes as in \cite[Lem.~3.6]{li1993algebraic}, we omit it. 
\end{proof}

\begin{lemma}
Let $E_1$ and $E_2$ be two $(n-1)$-polystable sheaves of class $c$ on $X$ such that $E_1^{[n-1]} \cong E_2^{[n-1]}$ and $l_{E_1} = l_{E_2}$. Then $E_1$ and $E_2$ define the same point in $M$.
\end{lemma}
\begin{proof}
Set $F \coloneqq E_1^{[n-1]} \cong E_2^{[n-1]}$, $\xi \coloneqq l_{E_1} = l_{E_2}$ and $l \coloneqq \length(\xi)$. Then $E_1$ and $E_2$ correspond to two closed points $y_1$ and $y_2$ respectively in the Quot scheme $\Quot(F,l)$. Denote by $\rho : \Quot(F,l) \to S^l(X)$ the Quot-to-Chow morphism that associates to any zero-dimensional quotient of $F$ of length $l$ its corresponding zero-cycle. By assumption $y_1$ and $y_2$ lie inside the fiber $\rho^{-1}(\xi)$. 

Suppose that $y_1$ and $y_2$ lie in the same connected component $C$ of $\rho^{-1}(\xi)$. Let $\sheaf{U}$ be the universal family of quotients over $\Quot(F,l)$, which fits in a short exact sequence
\begin{align*}
  0 \to \sheaf{K} \to \sheaf{O}_{\Quot} \otimes F \to \sheaf{U} \to 0
\end{align*}
such that $\sheaf{K}$ is flat over $\Quot(F,l)$. As $F$ is $(n-1)$-semistable, it follows that $\sheaf{K}|_C$ is a $C$-flat family of $(n-1)$-semistable sheaves of class $c$ on $X$. Choose a general divisor $D \in \Pi_a$ avoiding the support of $\xi$ such that $F|_D$ is Gieseker-polystable, which exists for $a \gg 0$ by Theorem~\ref{thm:restrictionStable}. Then there exists a short exact sequence
\begin{align}\label{eq:shortK}
  0 \to \sheaf{K}|_C(-D) \to \sheaf{K}|_C \to \sheaf{K}|_{C \times D} \to 0 
\end{align}
such that $\sheaf{K}|_{C \times D}$ is a $C$-flat family of fiber $F|_D$, cf. \cite[Thm.~22.5]{matsumura1989commutative}. This induces a constant morphism $\Psi_\sheaf{K} : C \to M_D$, where $M_D$ denotes the Simpson moduli space of Gieseker-semistable sheaves of Hilbert polynomial $P(F|_D)$ on $D$. As $w|_D \in c^{\perp} \cap \{1,h|_D,\ldots, h|_D^{n-1}\}^{\perp \perp}$, there exists a natural determinant line bundle $\lambda(w|_D)$ over $M_D$ such that
\begin{align*}
  \Psi_\sheaf{K}^*(\lambda(w|_D)) \cong \lambda_{\sheaf{K}|_{C \times D}}(w|_D) \cong \lambda_{\sheaf{K}|_C}(w \cdot h),
\end{align*}
see \cite[Thm.~8.1.5]{huybrechts2010geometry}. The last isomorphism above follows by applying the map $\lambda_{(-)}(w)$ to the short exact sequence \eqref{eq:shortK} and using \cite[Lem.~8.1.2, i)]{huybrechts2010geometry}. But since $\Psi_\sheaf{K}$ is constant, we obtain that $\lambda_\sheaf{K|_C}(w \cdot h)$ is trivial. Therefore, the classifying morphism $C \to M$ corresponding to the family $\sheaf{K}|_C$, as given by Theorem~\ref{thm:mainThm}, is constant. This shows that $E_1$ and $E_2$ define the same point in $M$ if they lie in the same connected component of $\rho^{-1}(\xi)$. Below we show that all fibers of $\rho$ are in fact connected, which will complete the proof.
\end{proof}

The proof of the following result is inspired by \cite{ellingsrud99irreducibility}.

\begin{theorem}
Let $F$ be a torsion-free sheaf on $X$ and $\xi \in S^l(X)$ a zero-cycle of length $l$. Then the fiber $\Quot_{\xi}(F,l) \coloneqq \rho^{-1}(\xi)$ over $\xi$ of the Quot-to-Chow morphism $\rho : \Quot(F,l) \to S^l(X)$ is connected.
\end{theorem}
\begin{proof}
The proof is by induction on $l$. It is enough to prove the statement when $\xi$ is supported at a single closed point of $X$, say $x \in X$. If $l = 1$, then $\xi$ corresponds to the closed point $x \in X$, and so $\Quot_{\xi}(F,1)$ is isomorphic to the projective space $\CP(F(x))$, with $F(x) = F_x \otimes k(x)$, which is clearly connected.

Now suppose that $\Quot_{\xi}(F,l)$ is connected for some $l > 0$. Let $\sheaf{U}$ be the universal family of quotients over $\Quot_{\xi}(F,l)$ and consider the short exact sequence
\begin{align*}
  0 \to \sheaf{K} \to \sheaf{O}_{\Quot} \otimes F \to \sheaf{U} \to 0
\end{align*}
over $\Quot_{\xi}(F,l)$. Then $Z \coloneqq \CP(\sheaf{K})$ is a projective space over $\Quot_{\xi}(F,l) \times X$. Note that $Z$ is connected, since $\Quot_{\xi}(F,l) \times X$ is so. 

Next we want to construct a surjective morphism $\varphi : Z \to \Quot_\xi(F,l+1)$. Let us first describe this map set-theoretically, at the level of geometric points.  By construction, a closed point of $Z$ corresponds to a quotient $[\mu : K \to k(x)] \in \CP(K)$, where $K$ is the kernel of some closed point $[F \to Q] \in \Quot_{\xi}(F,l)$. If $K' \coloneqq \Ker(\mu)$, then $F/K'$ is a zero-dimensional quotient of $F$ of length $l+1$ on $X$. We send $[K \to k(x)] \in Z$ through $\varphi$ to the quotient $[F \to F/K'] \in \Quot_\xi(F,l+1)$. In fact, one can do this construction in families, as described in \cite[Prop.~5]{ellingsrud99irreducibility}, to define an algebraic morphism $\varphi : Z \to \Quot_\xi(F,l+1)$. It remains to show that $\varphi$ is surjective. Let $[q: F \to T] \in \Quot_\xi(F,l+1)$ be a closed point and denote by $K \coloneqq \Ker(q)$ its kernel. Let $M \subset T_x$ be the subsheaf of all elements in $T_x$ annihilated by the maximal ideal in $\sheaf{O}_{X,x}$ and choose a point $[\lambda : k(x) \to M] \in \CP(M^{\vee})$. This induces a diagram as follows
\begin{center}
\begin{tikzcd}[column sep=normal]
            & 0                        & 0           &             & \\
0 \arrow[r] & k(x) \arrow[u] \arrow[r,"\lambda"] & T \arrow[u] \arrow[r] & Q \arrow[r] & 0 \\
0 \arrow[r] & K \arrow[u,"\mu"'] \arrow[r] & F \arrow[u,"q"'] \arrow[r] & Q \arrow[u,equal] \arrow[r] & 0 
\end{tikzcd}
\end{center}
such that $[F \to Q]$ is a zero-dimensional quotient of length $l$ on $X$. Then $[\mu : K \to k(x)] \in Z$ is mapped to $[q: F \to T]$ by $\varphi$, which shows that $\varphi$ is surjective. As $Z$ is connected, we conclude that $\Quot_\xi(F,l+1)$ must also be connected. 
\end{proof}